\documentclass[oneside,11pt]{amsart}
\usepackage{amsthm}
\usepackage{amssymb}
\usepackage{mathtools}

\newcommand{\N}{\mathbb{N}}
\newcommand{\Z}{\mathbb{Z}}

\newcommand{\D}{\mathbb{D}}

\newcommand{\K}{\mathbb{K}}
\newcommand{\ad}{\mathrm{ad}\,}

\newcommand{\End}{\mathrm{End}}
\newcommand{\id}{\mathrm{id}}

\newcommand{\NA}{\mathfrak{B}}

\newcommand{\Kspan}{\mathrm{span}_{\K}\,}
\newcommand{\supp}{\mathrm{supp}\,}

\newcommand{\trid}{\triangleright}
\newcommand{\ydG}{\prescript{G}{G}{\mathcal{YD}}}

\makeatletter
\numberwithin{equation}{section}
\numberwithin{figure}{section}
\numberwithin{table}{section}
\theoremstyle{plain}
\newtheorem{thm}{Theorem}[section]
\newtheorem*{thm*}{Theorem}
\newtheorem{lem}[thm]{Lemma}
\newtheorem{cor}[thm]{Corollary}

\newtheorem{pro}[thm]{Proposition}

\theoremstyle{remark}
\newtheorem{rem}[thm]{Remark}

\makeatother

\begin{document}

\title[Nichols algebras with finite root system of rank two]{Nichols algebras over groups\\
with finite root system of rank two III}

\author{I. Heckenberger}
\address{Philipps-Universit\"at Marburg\\ 
FB Mathematik und Informatik \\
Hans-Meerwein-Stra\ss e\\
35032 Marburg, Germany}
\email{heckenberger@mathematik.uni-marburg.de}

\author{L. Vendramin}
\address{Departamento de Matem\'atica, FCEN, Universidad de Buenos Aires, Pabell\'on 1, Ciudad Universitaria (1428), Buenos Aires, Argentina}
\email{lvendramin@dm.uba.ar}

\begin{abstract}
	We compute the finite-dimensional Nichols algebras over the sum of two simple
	Yetter-Drinfeld modules $V$ and $W$ over non-abelian epimorphic images of a certain
	central extension of the dihedral group of eight elements or
	$\mathbf{SL}(2,3)$, and such that the Weyl groupoid of the pair $(V,W)$ is
	finite. These central extensions appear in the classification of
	non-elementary finite-dimensional Nichols algebras with finite Weyl groupoid
	of rank two. We deduce new information on the structure of primitive elements of
  finite-dimensional Nichols algebras over groups.
\end{abstract}

\maketitle

\setcounter{tocdepth}{1}
\tableofcontents{}

\section*{Introduction}

In \cite{MR2732989} an approach to a particular instance of the classification
of finite-dimensional Nichols algebras was initiated. Assume that $U$ is the
direct sum of two absolutely simple Yetter-Drinfeld modules $V$ and $W$ and
that $G$ is generated by the support of $U$. If the Nichols algebra of $U$ is
finite-dimensional, then the Weyl groupoid of $(V,W)$ is finite, and this
groupoid can be calculated, see \cite{MR2766176}, \cite{MR2390080},
\cite{MR2525553}, \cite{MR2734956} and \cite{MR2732989}. 
If the square of the braiding between $V$ and $W$
is the identity, then by a result of Gra\~na the Nichols algebra of $U$
is the tensor product of the Nichols algebras of $V$ and $W$. So we are
interested in the remaining cases. For a particular class of groups and
Yetter-Drinfeld modules, it was possible to construct and classify those $U$
with a finite-dimensional Nichols algebra.

A breakthrough for the approach in \cite{MR2732989} was achieved
in \cite[Thm.\,4.5]{partII}, where it was proved that if
the square of the braiding between $V$ and $W$ is not the identity and
$\NA(V\oplus W)$ is finite-dimensional, then $G$ is a non-abelian quotient of
one of four groups which can be described explicitly. Again, the main tool was
the theory of Weyl groupoids of tuples of simple Yetter-Drinfeld modules.

The theorem in \cite{partII} has very far reaching consequences. One of these
consequences is the existence of good bounds for the dimensions of $V$ and $W$,
see \cite[Cor.  4.6]{partII}.  Another consequence is the possibility to obtain
new examples of finite-dimensional Nichols algebras.  This can be done if the
Weyl groupoids (and of course the Nichols algebras) appearing in the context of
\cite{partII} are studied. 

In this work we study finite-dimensional Nichols algebras over two of the
groups appearing in \cite[Thm.\,4.5]{partII}. One of these groups is the
so-called group $T$, which is a certain central extension of the group
$\mathbf{SL}(2,3)$, see Section~\ref{subsection:T:preliminaries}.  The other is
$\Gamma_4$, a central extension of the dihedral group of eight elements, see
Section~\ref{subsection:G4:preliminaries}.  To study these
Nichols algebras we recognize some pairs $(V,W)$ of Yetter-Drinfeld modules
over non-abelian epimorphic images of $T$ and $\Gamma_4$ admitting a Cartan
matrix of finite type.  Then we determine when the Yetter-Drinfeld modules
$(\ad V)^m(W)$ and $(\ad W)^m(V)$ are absolutely simple or zero for all
$m\in\N$. In fact, around one third of the paper consists of these calculations,
which form the most technical but highly important part of this work. 
With the results of the calculations in the pocket,
we can compute the Cartan matrices, the reflections and the Weyl groupoid
of the pairs $(V,W)$.
This allows us to determine the
structure of Nichols algebras over non-abelian epimorphic images of $T$ and
$\Gamma_4$.

As a consequence, we will obtain
two (new families of) finite-dimensional Nichols algebras, see Theorems
\ref{thm:T} and \ref{thm:G4}.  One of these families of Nichols algebras has a
root system of type $G_2$ and dimension 
\[
	\begin{cases}
		6^3\,72^3 & \text{if $\mathrm{char}\,\K\ne2$,}\\
		3^3\,36^3 & \text{if $\mathrm{char}\,\K=2$}.
	\end{cases}
\]
The others have a root system of type $B_2$ and dimension
\[
    \begin{cases}
	    8^264^2 & \text{if $\mathrm{char}\,\K\ne2$},\\
   		4^264^2 & \text{if $\mathrm{char}\,\K=2$}.
	\end{cases}
\]

As a byproduct of our study of the Nichols algebras associated to epimorphic
images of $T$ and $\Gamma_4$ we improve the application given in \cite[Cor.
4.6]{partII}.  More precisely, under the assumptions of
\cite[Thm.\,4.5]{partII} we conclude that the support of the sum of the two
simple Yetter-Drinfeld modules is isomorphic (as a quandle) to one of the five
quandles listed in Theorem \ref{thm:main}. 

This work and the results of \cite{MR2732989} and \cite{partII} are an
important part of the classification of Nichols algebras admitting a finite
root system of rank two achieved in \cite{rank2}.

The paper is organized as follows. Section \ref{section:braiding} is devoted to
state some general facts about adjoint actions and braidings. In Section
\ref{section:basics:T} we review basic facts about the group $T$ and state the
main results concerning Nichols algebras over non-abelian epimorphic images of
$T$, see Proposition \ref{pro:deg=1} and Theorem \ref{thm:T}. These results are
proved in Sections \ref{section:deg=1} and \ref{section:T}.  In Section
\ref{section:basics:G4} we review the basic facts concerning the group
$\Gamma_4$ and state our main result about Nichols algebras over non-abelian
epimorphic images of $\Gamma_4$, see Theorem \ref{thm:G4}. This theorem is then
proved in Section \ref{section:G4}.  Finally, in Section
\ref{section:applications} the application mentioned in the previous paragraph
is deduced.

\section{Some preliminaries}
\label{section:braiding}

Fix a field $\K$. We use the notations and the definitions given in
\cite[Section 2.1]{partII} mostly without recalling them again.  However, we
recall the definition of the Cartan matrix of a pair of 
Yetter-Drinfeld modules.
Let $G$ be a group and let $V,W\in \ydG $.  If 
$(\ad V)^p(W)=0$ and $(\ad W)^{q}(V)=0$ 
for some $p,q\in\N_0$ 
then one defines
the \emph{Cartan matrix} $(a_{ij}^{M})\in\Z^{2\times2}$ of $M$ by
\begin{align*}
	&a_{11}^M=a_{22}^{M}=2,\\
	&a_{12}^M=-\sup\{m\in\N_0:(\ad V)^m(W)=0\},\\
	&a_{21}^M=-\sup\{m\in\N_0:(\ad W)^m(V)=0\}.
\end{align*}

In \cite{partII}, a sufficient criterion for the non-vanishing of 
\[
(\ad W)^m(V)\subseteq \NA (V\oplus W)
\]
for $m\in \N $ was formulated in terms of
elements of $G$ satisfying some properties. We can use this idea to obtain a
condition on the braiding of $V\oplus W$ under some assumptions on $(\ad
W)^m(V)$ and $(\ad W)^{m+1}(V)$ for some $m\in \N $. The following proposition
(and its proof) is analogous to \cite[Prop.\,5.5]{partII}. Before reading it,
we strongly recommend to read \cite[Prop.\,5.5]{partII} and its proof.

\begin{pro}
  \label{pro:degrees}
  Let $G$ be a group and let $V,W\in \ydG $. Let $m\in \N $, $i\in \{1,\dots ,m\}$,
  $r_1,\dots,r_m,p_1,\dots ,p_m\in \supp W$ and $s,p_{m+1}\in \supp V$. Assume that
  $(p_1,\dots ,p_{m+1})\in \supp Q_m(r_1,\dots ,r_m,s)$,
  $Q_{m+1}(p_i,r_1,\dots ,r_m,s)=0$, and that
  \begin{align}
    &p_{i+1}\trid p_i\ne p_i,\quad p_j\trid p_i=p_i
    \text{ for all $j$ with $i+1<j\le m+1$,}\\
    &p_i\notin \{p_j\,|\,1\le j\le m,j\not=i\} \cup
    \{ (p_{j+1}\cdots p_{m+1})^{-1}\trid p_j\,|\,1\le j<i\}.
  \end{align}
  Then $\dim W_{p_i}=1$ and $p_iw=-w$ for all $w\in W_{p_i}$.
\end{pro}

\begin{proof}
	By definition of $Q_{m+1}$, the set $\supp Q_{m+1}(p_i,r_1,\dots ,r_m,s)$
	consists of tuples of the form
  \begin{equation} \label{eq:suppQm+1}
    \begin{aligned}
      &(p_i\trid p'_1,\dots ,p_i\trid p'_{j-1},p_i,p'_j,\dots ,p'_{m+1}),\\
      &(p_i\trid p'_1,\dots ,p_i\trid p'_{j-1},p_ip'_j\cdots p'_{m+1}\trid p_i,
      p_i\trid p'_j,\dots ,p_i\trid p'_{m+1})
    \end{aligned}
  \end{equation}
	with $1\le j\le m+1$, where $(p'_1,\dots ,p'_{m+1})\in \supp Q_m(r_1,\dots
	,r_m,s)$.  This for $j=i$ and the assumption imply that \[
		(p_i\trid p_1,\dots,p_i\trid p_{i-1},p_i,p_i,p_{i+1},\dots ,p_{m+1})
	\]
	appears among the tuples
	in \eqref{eq:suppQm+1}. Comparing this tuple with all other possible tuples
	similarly to the proof of \cite[Prop.\,5.5]{partII}, one obtains that it
	appears precisely twice among the tuples in \eqref{eq:suppQm+1}: In the first
	line for $j=i$ and for $j=i+1$, where $p_k'=p_k$ for all $k\in \{1,2,\dots
	,m+1\}$.  The two tuples correspond to the summands
  $$c_{i-1\,i}\cdots c_{23}c_{12}(\id \otimes \pi _{p_1}\otimes \cdots \otimes \pi _{p_m}\otimes \pi _{p_{m+1}})
  (u)$$
  and
  $$c_{i\,i+1}\cdots c_{23}c_{12}(\id \otimes \pi _{p_1}\otimes \cdots \otimes \pi _{p_m}\otimes \pi _{p_{m+1}})
  (u) $$
  of $\varphi _{m+1}(u)$ for any $u\in W_{p_i}\otimes Q_m(r_1,\dots ,r_m,s)$.
  As $Q_{m+1}(p_i,r_1,\dots ,r_m,s)=0$ by assumption, we obtain that
  $$(\id +c_{i,i+1})
  c_{i-1\,i}\cdots c_{23}c_{12}(\id \otimes \pi _{p_1}\otimes \cdots \otimes \pi _{p_m}\otimes \pi _{p_{m+1}})(u)=0$$
  for all
  $u\in W_{p_i}\otimes Q_m(r_1,\dots ,r_m,s)$.
	Then there exists $w_0\in W_{p_i}\setminus \{0\}$ such that $(\id
	+c)(w\otimes w_0)=0$ for all $w\in W_{p_i}$.  Since $c(w\otimes
	w_0)=p_iw_0\otimes w$ for all $w\in W_{p_i}$, we conclude that $\dim
	W_{p_i}=1$ and $p_iw=-w$ for all $w\in W_{p_i}$.
\end{proof}

\begin{lem}
    \label{lem:c^2}
    Let $G$ be a group and let $V,W\in \ydG$. Assume that there exist
    $q_V,q_W\in\K$ such that
    $xv=q_Vv$ and $yw=q_Ww$ 
    for all $x,y\in G$, $v\in V_x$, $w\in W_y$. 
    Then 
    \[
        c_{W,V}c_{V,W}(v\otimes w)=q_V^{-1}xy v\otimes q_W^{-1}xy w
    \]
    for all $x,y\in G$, $v\in V_x$, $w\in W_y$. 
\end{lem}

\begin{proof}
    A direct computation yields 
    \[
        c_{W,V}c_{V,W}(v\otimes w)=c_{W,V}(xw\otimes v)=xyx^{-1}v\otimes xw.
    \]
    Since $x^{-1}v=q_V^{-1}v$ and $xyw=q_Wxw$, the lemma
    follows. 
\end{proof}

One of the key step towards our main result depends on the
calculation of the Yetter-Drinfeld modules $(\ad V)^m(W)$ and $(\ad W)^m(V)$
for some Yetter-Drinfeld modules $V$ and $W$ and for all $m\in\N$.  
For that purpose, the following lemma is useful.

\begin{lem}{\cite[Thm.~1.1]{MR2732989}}
  \label{lem:X_n}
  Let $V$ and $W$ be Yetter-Drinfeld modules over a Hopf algebra $H$ with
  bijective antipode.  Let $\varphi _0=0$ and  
  $\varphi_m\in\End (V^{\otimes m}\otimes W)$ be given by 
  \begin{align*}
		\varphi_m &= \id-c_{V^{\otimes(m-1)}\otimes W,V}\,c_{V,V^{\otimes(m-1)}
		\otimes W}+(\id\otimes\varphi_{m-1})c_{1,2}
  \end{align*}
  for all $m\geq1$, and let  
  $X_0^{V,W}=W$, and 
  \[
  	X_m^{V,W} = \varphi_m(V\otimes X_{m-1})\subseteq V^{\otimes m}\otimes W
  \]
  for all $m\geq1$.
  Then $(\ad V)^n(W)\simeq X_n^{V,W}$ for all $n\in \N _0$.
\end{lem}

In the paper, we will use the Yetter-Drinfeld modules $X_m^{V,W}$ for
calculations, but in the applications we usually turn back to the more
suggestive module $(\ad V)^m(W)$.

The basic theory of Weyl groupoids and Nichols algebras of
\cite{MR2766176} and \cite{MR2734956} is reviewed in 
\cite[Section 2]{MR2732989}.

\section{Nichols algebras over epimorphic images of $T$}
\label{section:basics:T}

\subsection{Preliminaries}
\label{subsection:T:preliminaries}

Recall that the group $T$ is 
\[
    T=\langle\zeta\rangle\times\langle\chi_1,\chi_2,\chi_3,\chi_4\mid\chi_i\chi_j=\chi_{i\triangleright j}\chi_i,\quad i,j\in\{1,2,3,4\}\rangle,
\]    
where $\triangleright$ is defined by
\begin{center}
	\begin{tabular}{c|cccc}
		$\triangleright$ & $1$ & $2$ & $3$ & $4$\tabularnewline
		\hline
		$1$ & $1$ & $4$ & $2$ & $3$\tabularnewline
		$2$ & $3$ & $2$ & $4$ & $1$\tabularnewline
		$3$ & $4$ & $1$ & $3$ & $2$\tabularnewline
		$4$ & $2$ & $3$ & $1$ & $4$\tabularnewline
	\end{tabular}
\end{center}
The table describes the structure of the quandle associated to the vertices of
the tetrahedron,
see \cite[\S1]{MR1994219}.  By \cite[Lemmas 2.17 and 2.18]{MR2803792}, 
\begin{equation}
	\label{eq:T_x13=x23=x33=x34}
	\chi_1^3=\chi_2^3=\chi_3^3=\chi_4^3
\end{equation}
is a central element of $T$. Moreover, the center of $T$ is $Z(T)=\langle
\chi_1^3,\chi_1\chi_2\chi_3, \zeta\rangle$. 

The group $T$ can be presented by generators $\chi_1$,
$\chi_2$, $\zeta$ with relations 
\begin{equation}
    \label{eq:presentation}
	\begin{gathered}
		\zeta\chi_1=\chi_1\zeta,\quad 
		\zeta\chi_2=\chi_2\zeta,\quad 
		\chi_1\chi_2\chi_1=\chi_2\chi_1\chi_2,\quad
		\chi_1^3=\chi_2^3.
	\end{gathered}
\end{equation}
Then $\chi_3=\chi_2\chi_1\chi_2^{-1}$ and $\chi_4=\chi_1\chi_2\chi_1^{-1}$ in
$T$, and the elements $\chi_1,\chi_2,\chi_3,\chi_4$ form a conjugacy class of
$T$.  The group $T$ is isomorphic to the enveloping group of the quandle
$\chi_1^T\cup\zeta^T$. 

\begin{rem}
	\label{rem:T+1}
	The $\chi_1^T\cup\zeta^T$ is the disjoint union of the trivial quandle with one
	element and the quandle associated to the vertices of the tetrahedron.
\end{rem}

In what follows, let $G$ be a non-abelian quotient of the group $T$.
Equivalently, the elements $\chi_i$, $1\le i\le 4$, represent pairwise different
elements of $G$.  Let $z\in Z(G)$ and $x_1\in G$ such that $G=\langle
z,x_1^G\rangle$ and there exists a quandle isomorphism
$f\colon\chi_1^T\cup\{\zeta\}\to x_1^G\cup\{z\}$ with $f(\chi_1)=x_1$.  For all
$2\leq i\leq4$ let $x_i\coloneqq f(\chi_i)$.  The quandle isomorphism $f$ induces a
surjective group homomorphism $T\to G$. 

\begin{lem}
    \label{lem:centralizers}
    The following hold:
    \begin{enumerate}
        \item $G^z=G$.
				\item $G^{x_1}=\langle x_1,x_2x_3,z\rangle$.
    \end{enumerate}
\end{lem}

\begin{proof}
    The first claim is trivial since $z$ is central. Let us prove (2). Since
    $T^{\chi_1}=\langle\chi_1,\chi_2\chi_3,\zeta\rangle$ by \cite[Lemma
    5.5]{MR2803792}, we obtain that $\langle x_1,x_2x_3,z\rangle\subseteq
		G^{x_1}$. Moreover, $T^{\chi_1}$ has index four in $T$. From $|x_1^G|=4$
		we conclude that $G^{x_1}$ has index four in $G$. Hence $G^{x_1}=\langle
    x_1,x_2x_3,z\rangle$.
\end{proof}

\begin{lem}
	\label{lem:sigma(x1)=-1}
    Let $V,W\in\ydG$ such that $\supp V=z^G$, $\supp W=x_1^G$, and
    $(\ad V)(W)\ne0$. Then $(\ad W)^3(V)\not=0$.
    If $(\ad W)^3(V)$ is irreducible, then $\dim W=4$ and
    $x_1w=-w$ for all $w\in W_{x_1}$.
\end{lem}

\begin{proof}
    First, $(x_4,z)\in \supp Q_1(x_4,z)$ since $c_{V,W}c_{W,V}\ne \id$.
    Therefore \cite[Prop.\,5.5]{partII} implies that
    $(x_1,x_4,z)\in\supp Q_2(x_1,x_4,z)$.
    Since \[
        x_2\not\in\{x_1,x_4,(x_4z)^{-1}\triangleright x_1=x_3\},
    \]
    \cite[Prop.\,5.5]{partII} with $i=2$ yields
    \[
        (x_2\triangleright x_1,x_2,x_4,z)\in\supp Q_3(x_2,x_1,x_4,z). 
    \]
    In particular, $(\ad W)^3(V)\not=0$. It is easy to check that
    $x_2x_1x_4=x_1x_2x_3$,
		and hence the non-central element $x_1x_1x_4$
		is not conjugate to the central element $x_2x_1x_4$
    in $G$. Therefore $Q_3(x_1,x_1,x_4,z)=0$ by the irreducibility of $(\ad
    W)^3(V)$, and hence Proposition~\ref{pro:degrees} with the parameters
		$m=2$, $i=1$, and
    $(p_1,p_2,p_3)=(r_1,r_2,s)=(x_1,x_4,z)$ yields the claim.
\end{proof}

Let $W=M(x_1,\sigma)$ be a Yetter-Drinfeld module over $G$ for some
absolutely irreducible representation $\sigma$ of $G^{x_1}$.  The centralizer
$G^{x_1}=\langle x_1,x_2x_3,z\rangle$ is abelian and hence $\deg\sigma=1$.  Let
$\epsilon =\sigma (x_2x_3)$.

\begin{rem} \label{rem:G_on_generalW}
    Let $w_1\in W_{x_1}$ such that $w_1\ne0$. Then $w_1$,
    $w_2\coloneqq\sigma(x_1)^{-1}x_4w_1$, $w_3\coloneqq\sigma(x_1)^{-1}x_2w_1$,
    $w_4\coloneqq\sigma(x_1)^{-1}x_3w_1$ is a basis of $W$. The degrees of these
    vectors are $x_1$, $x_2$, $x_3$ and $x_4$, respectively. Furthermore, 
    $x_iw_j=q_{ij}w_{i\triangleright j}$, where 
	\[
	q_{ij}=\left(\begin{array}{cccc}
		\sigma(x_{1}) & \sigma(x_{1}) & \sigma(x_{1}) & \sigma(x_{1})\\
    \sigma(x_{1}) & \sigma(x_{1}) & \sigma(x_{1})^{3}\epsilon ^{-1}& \sigma(x_{1})^{-1}\epsilon\\
		\sigma(x_{1}) & \sigma(x_{1})^{-1}\epsilon & \sigma(x_{1}) &
    \sigma(x_{1})^{3}\epsilon^{-1}\\
    \sigma(x_{1}) & \sigma(x_{1})^{3}\epsilon^{-1} & \sigma(x_{1})^{-1}\epsilon & \sigma(x_{1})
	\end{array}\right).
	\]
    For example, one can easily compute that 
    \[
        x_2w_4=\sigma(x_1)^{-1}x_2x_3w_1=\sigma(x_1)^{-1}\epsilon w_1.
    \]
	Then, since $x_1^3=x_2^3$ is central, 
    \begin{align*}
        x_2w_3 &= \sigma(x_1)^{-1}x_2x_2w_1=\sigma(x_1)^{-1}x_2^{-1}x_2^3w_1\\
        &=\sigma(x_1)^{-1}x_2^{-1}x_1^3w_1=\sigma(x_1)^2x_2^{-1}w_1
        =\sigma(x_1)^3\epsilon ^{-1}w_4.
    \end{align*}
\end{rem}

\begin{rem}
  \label{rem:G_on_W}
  Assume that $\sigma (x_1)=-1$. 
  Since $x_1^4=(x_2x_3)^2$, we obtain that $\epsilon ^2=1$.  
    Then the action of $G$ on $W$ is given by the following table:
	\begin{center}
		\begin{tabular}{c|cccc}
			$W$ & $w_{1}$ & $w_{2}$ & $w_{3}$ & $w_{4}$\tabularnewline
			\hline
			$x_{1}$ & $-w_{1}$ & $-w_{4}$ & $-w_{2}$ & $-w_{3}$\tabularnewline
			$x_{2}$ & $-w_{3}$ & $-w_{2}$ & $-\epsilon w_{4}$ & $-\epsilon w_{1}$\tabularnewline
			$x_{3}$ & $-w_{4}$ & $-\epsilon w_{1}$ & $-w_{3}$ & $-\epsilon w_{2}$\tabularnewline
			$x_{4}$ & $-w_{2}$ & $-\epsilon w_{3}$ & $-\epsilon w_{1}$ & $-w_{4}$\tabularnewline
			$z$ & $\sigma(z) w_1$ & $\sigma(z) w_2$ & $\sigma(z) w_3$ & $\sigma(z) w_4$\tabularnewline 
		\end{tabular}
	\end{center}
\end{rem}

Let $V=M(z,\rho)$ for some absolutely irreducible representation $\rho$ of the
centralizer $G^z=G$. The following lemma tells us that we will only need to
study those representations of degree at most two.

\begin{lem}
	\label{lem:degrees}
	Assume that $(\ad W)(V)\subseteq \NA (V\oplus W)$ and $(\ad W)^2(V)$ are
	absolutely simple Yetter-Drinfeld modules over $G$. Then $\dim V\leq2$.
\end{lem}

\begin{proof}
	By \cite[Lemma 1.7]{MR2732989}, 
    \begin{equation*}
		X_1^{W,V}=\varphi_1(W\otimes V)=
        \K G\{\varphi_1(w_1\otimes v)\mid v\in V\}.
    \end{equation*}
	Moreover, a direct computation yields
	\begin{align}
		\varphi_1(w_i\otimes v)
    =w_i\otimes v-c_{V,W}c_{W,V}(w_i\otimes v)=w_i\otimes(v-\sigma(z)x_iv)
	\end{align}
    for $i\in\{1,2,3,4\}$ and $v\in V$.  Since $X_1^{W,V}\simeq (\ad
    W)(V)\ne0$, there exists $v\in V$ such that the tensor
    $w_1\otimes(v-\sigma(z)x_1v)\in(W\otimes V)_{x_1z}$ is non-zero.  Let
    $v_0\coloneqq v-\sigma(z)x_1v$. Since $X_1^{W,V}$ is absolutely simple and the
    centralizer of $x_1z$ is abelian, $(X_1^{W,V})_{x_1z}$ is one-dimensional.
    Therefore there exist $\alpha_1,\alpha_2\in\K^\times$ such that 
	\begin{align}
		\label{eq:1}
		&\alpha_1^4=\alpha_2^2,&& x_1v_0=\alpha_1 v_0,&& x_2x_3v_0=\alpha_2v_0.
	\end{align}
	By \cite[Lemma 1.7]{MR2732989},
    \[
		X_2^{W,V}=\varphi(W\otimes X_1^{W,V})=\K G\{\varphi_2(w_1\otimes
    w_1\otimes v_0),\varphi_2(w_2\otimes w_1\otimes v_0)\}.
    \]
	Let $y\coloneqq\varphi_2(w_2\otimes w_1\otimes v_0)\in X_2^{W,V}$.  A direct calculation
    yields
    \begin{multline*}
        y=\varphi_2(w_2\otimes w_1\otimes v_0) = w_2\otimes w_1\otimes v_0 - x_3zw_2\otimes x_2w_1\otimes x_2v_0\\
        +x_2w_1\otimes w_2\otimes (v_0-\sigma(z)x_2v_0),
    \end{multline*}
    and hence $y\in(W\otimes W\otimes V)_{x_2x_1z}$ is non-zero.  
    Since $(\ad W)^2(V)$ is
    absolutely simple and the centralizer of $x_2x_1z$ is the
    abelian group \[
			G^{x_2x_1z}=x_3G^{x_2x_3z}x_3^{-1}=\langle x_2x_1z, x_4,z\rangle,
		\]
		there exists $\xi\in\K$ such that $x_4y=\xi y$.  The second
    tensor factors $w_1$, $x_2w_1$, and $w_2$ in $y$ are linearly independent
    and $2\trid 1=3$.  Hence, by comparing the third tensor factors, we
		conclude that there exists $\alpha_3\in\K\setminus\{0\}$ such that 
    \begin{equation}
        \label{eq:2}
        \alpha_3(v_0-\sigma(z)x_2v_0)=x_4v_0.
    \end{equation}
    By the presentation for $T$ given in \eqref{eq:presentation}
    and by the irreducibility of $V$, it is enough
    to show that $S\coloneqq\Kspan\{v_0,x_2v_0\}$
    is stable under the action of $x_1$ and $x_2$.
    First, $x_1v_0\in S$ since $x_1v_0=\alpha_1v_0$. Equations $x_1x_2=x_4x_1$ and
    \eqref{eq:2} imply that $x_1x_2v_0=x_4x_1v_0=\alpha_1x_4v_0\in S$. Finally,
    applying $x_2$ to Equation \eqref{eq:2} and using $x_2x_4=x_4x_1$ we
    conclude that $x_2^2v_0\in S$. 
\end{proof}

\begin{lem}
	\label{lem:deg2}
    Assume that $\K $ is algebraically closed.
    Let $(\rho,U)$ be an irreducible representation of $\K G$ of degree $2$.
    Then $\mathrm{char}(\K)\ne2$, $\rho(z)\in\K^\times$
    and there exist $\alpha,\beta\in\K^\times$ with $\beta^2+\beta+1=0$
    and a basis of $U$ such that 
		\begin{equation}
		\label{eq:rho}
		\begin{aligned}
			&\rho(x_{1})=\left(
			\begin{array}{cc}
				\alpha & -\alpha^{2}\beta^{2}\\
				0 & \alpha\beta
			\end{array}\right),
			\quad
			&&\rho(x_{2})=\left(
			\begin{array}{cc}
				0 & -\alpha^{2}\beta\\
				1 & -\alpha\beta^{2}
			\end{array}\right),\\
			&\rho(x_{3})=\left(
			\begin{array}{cc}
				\alpha\beta & 0\\
				\beta^{2} & \alpha
			\end{array}\right),
			\quad
			&&\rho(x_{4})=\left(
			\begin{array}{cc}
				-\alpha\beta^{2} & -\alpha^{2}\\
				\beta & 0
			\end{array}\right),
		\end{aligned}
	\end{equation}
    with respect to this basis. Further, $\rho (x_1x_2x_3)=-\alpha^3\id _U$.
\end{lem}

\begin{proof}
	Let $v_0\in U\setminus\{0\}$ and let $\alpha_1,\alpha_2\in\K^\times$ such
	that $x_1v_0=\alpha_1v_0$ and $x_2x_3v_0=\alpha_2v_0$. Then
	$\alpha_1^4=\alpha_2^2$. Since $\deg\rho=2$ and $G$ is generated by $x_1$,
	$x_2$, and the central element $z$, $U=\Kspan\{v_0,x_2v_0\}$ and
	$x_3v_0=\beta_1v_0+\beta_2x_2v_0$ and $x_4v_0=\beta_3v_0+\beta_4x_2v_0$ for
	some $\beta_1,\beta _2,\beta _3,\beta_4\in\K$. Writing
	$x_2x_3v_0=\alpha_2v_0$ as $\alpha_2^{-1}x_3v_0=x_2^{-1}v_0$ and using
	\eqref{eq:T_x13=x23=x33=x34}
	we conclude that
	\[
		x_2(x_2v_0)=x_2^{-1}x_1^3v_0
		=\alpha_1^3\alpha_2^{-1}x_3v_0=\alpha_1^3\alpha_2^{-1}(\beta_1v_0+\beta_2x_2v_0).
	\]
	Therefore 
	\begin{equation*}
		\begin{aligned}
			&\rho(x_{1})=\left(
			\begin{array}{cc}
				\alpha_1 & \alpha_1\beta_3\\
				0 & \alpha_1\beta_4
			\end{array}\right),
			\quad
			&&\rho(x_{2})=\left(
			\begin{array}{cc}
				0 & \alpha_1^3\alpha_2^{-1}\beta_1\\
				1 & \alpha_1^3\alpha_2^{-1}\beta_2
			\end{array}\right),\\
			&\rho(x_{3})=\left(
			\begin{array}{cc}
				\beta_1 & 0\\
				\beta_2 & \alpha_1
			\end{array}\right),
			\quad
			&&\rho(x_{4})=\left(
			\begin{array}{cc}
				\beta_3 & \alpha_2\\
				\beta_4 & 0
			\end{array}\right).\\
		\end{aligned}
	\end{equation*}
	Since $\det\rho(x_1)=\det\rho(x_3)$ and $x_3x_2=x_1x_3$, we obtain that 
  $\alpha_1\beta_4=\beta _1$ and
	\begin{align*}
		&\beta_2\beta_4=1,&&\beta_2\beta_3+\beta_1=0,
    &&\alpha _1\alpha_2^{-1}\beta_1^2=\beta_3,
    &&\alpha_1\alpha_2^{-1}\beta_2(\alpha_1+\beta_1)=\beta_4.
	\end{align*}
	Let $\alpha\coloneqq\alpha_1$ and $\beta\coloneqq\beta_4$. Then the above
  equations are equivalent to
	\begin{align*}
		&\alpha_2=-\alpha^2,&&\beta_1=\alpha\beta,&&\beta_2=\beta^2,
    &&\beta_3=-\alpha\beta^2,&&\beta^2+\beta+1=0.
  \end{align*}
  Hence we conclude \eqref{eq:rho}. Since $\rho
  (x_1x_2x_3)v_0=\alpha_1\alpha _2v_0=-\alpha ^3v_0$, we obtain that
  $\rho (x_1x_2x_3)=-\alpha ^3\id _U$ from $x_1x_2x_3\in Z(G)$,
  the absolute irreducibility of $\rho $, and Schur's Lemma.

  Assume that $\mathrm{char}\,\K=2$. Then $v=\alpha v_0+x_2v_0\in U$ is a
  $\rho$-invariant vector. This is a contradiction to the irreducibility of
  $(\rho ,U)$.
\end{proof}
 
\subsection{Main results}

Let $G$, $z,x_1,\dots,x_4$, $V$ and $W$ as in Subsection
\ref{subsection:T:preliminaries}.  Our aim is to prove Proposition
\ref{pro:deg=1} and Theorem \ref{thm:T} below.  

\begin{pro}
	\label{pro:deg=1}
    Let $V=M(z,\rho)$ and $W=M(x_1,\sigma)$ be absolutely simple
    Yetter-Drinfeld modules over $G$.  Assume that $(V,W)$ admits all
    reflections, the Weyl groupoid $\mathcal{W}(V,W)$ is finite, and the Cartan
    matrix of $(V,W)$ is non-diagonal and of finite type.  Then $\deg\rho=1$. 
\end{pro}

The proof of Proposition \ref{pro:deg=1} will be given in Section
\ref{section:deg=1}.  

Recall that $(k)_t=1+t+\cdots+t^{k-1}$ for all $k\in \N$.

\begin{thm}
	\label{thm:T}
	Let $V=M(z,\rho)$ and $W=M(x_1,\sigma)$ be absolutely simple
	Yetter-Drinfeld modules over $G$. 
  Assume that $c_{W,V}c_{V,W}\not=\id_{V\otimes W}$. The following
	are equivalent:
	\begin{enumerate}
		\item The Nichols algebra $\NA (V\oplus W)$ is finite-dimensional.
		\item The pair $(V,W)$ admits all reflections and $\mathcal{W}(V,W)$ is
			finite.
		\item $\deg\rho=1$, and
      $(\rho (x_1)\sigma (z))^2-\rho (x_1)\sigma (z)+1=0$,
      $\sigma (x_1)=-1$, $\sigma (x_2x_3)=1$, $\rho (x_1z)\sigma (z)=1$.
	\end{enumerate}
	In this case, $\mathcal{W}(V,W)$ is standard with Cartan
	matrix of type $G_2$.  If $\mathrm{char}\,\K\ne 2$ then
    \begin{multline*}
        \mathcal{H}_{\NA (V\oplus W)}(t_1,t_2)\\=(6)_{t_1}(6)_{t_1t_2^3}(6)_{t_1^2t_2^3}(2)_{t_2}^2(3)_{t_2}(6)_{t_2}
        (2)_{t_1t_2}^2 
        (3)_{t_1t_2}(6)_{t_1t_2}(2)_{t_1t_2^2}^2(3)_{t_1t_2^2}(6)_{t_1t_2^2}, 
    \end{multline*}
	and $\dim \NA (V\oplus W)=6^3\,72^3=80621568$.
	If $\mathrm{char}\,\K=2$ then
	\begin{align*}
        \mathcal{H}_{\NA (V\oplus W)}(t_1,t_2)=(3)_{t_1}(3)_{t_1t_2^3}(3)_{t_1^2t_2^3}(2)_{t_2}^2(3)_{t_2}^2
        (2)_{t_1t_2}^2(3)_{t_1t_2}^2 (2)_{t_1t_2^2}^2(3)_{t_1t_2^2}^2, 
	\end{align*}
	and $\dim \NA (V\oplus W)=3^3\,36^3=1259712$.
\end{thm}

We will prove Theorem \ref{thm:T} in Section \ref{section:T}. 

\section{Proof of Proposition \ref{pro:deg=1}}
\label{section:deg=1}

Let $V=M(z,\rho)$ and $W=M(x_1,\sigma)$
as in Subsection \ref{subsection:T:preliminaries}.
We write $X_n=X_n^{V,W}$ and $\varphi_n=\varphi_n^{V,W}$
for all $n\in\N_0$ if no confusion can arise.
We now prepare the proof of Proposition~\ref{pro:deg=1}.
Assume that $\deg\rho=2$, $\rho$ is given by
\eqref{eq:rho} of Lemma~\ref{lem:deg2} with respect to
a basis $\{v_0,x_2v_0\}$ of $V$, and that
the characteristic of $\K$ is not $2$.

Assume that $\sigma$ is an absolutely
irreducible representation of $G^{x_1}$ with
$\sigma(x_1)=-1$. Then $\sigma(x_2x_3)^2=1$. The action of $G$ on $W$ is
described in Remark~\ref{rem:G_on_W}.
We first compute $(\ad V)(W)\simeq X_1^{V,W}$. By
\cite[Lemma\,1.7]{MR2732989}, 
\[
    X_1^{V,W}=\varphi_1(V\otimes W)
	=\K G\{\varphi_1(v_0\otimes w_1),\varphi_1(x_2v_0\otimes w_1)\}.
\]
We record explicit formulas for later use in the following lemma.

\begin{lem}
	\label{lem:deg2:X1_auxiliar}
	Assume that $\sigma(x_1)=-1$. Then the following hold:
	\begin{align}
		\label{eq:deg2:phi1(v0,w1)} &\varphi_1(v_0\otimes w_1) = 
			(1-\sigma(z)\alpha) v_0\otimes w_1,\\
		\label{eq:deg2:phi1(x2v0,w1)} &\varphi_1(x_2v_0\otimes w_1)
			=(1-\sigma(z)\alpha\beta)x_2v_0\otimes
      w_1+\sigma(z)\alpha^2\beta^2v_0\otimes w_1.
	\end{align}
	Further $w_1'\coloneqq\varphi_1(x_2v_0\otimes w_1)\in(V\otimes W)_{x_1z}$ is
  non-zero and hence $X_1^{V,W}\ne0$.
\end{lem}

\begin{proof}
	Equation \eqref{eq:deg2:phi1(v0,w1)} follows by a direct computation using
  Remark~\ref{rem:G_on_W} and \eqref{eq:rho}. Let us prove Equation
	\eqref{eq:deg2:phi1(x2v0,w1)}. Using Remark~\ref{rem:G_on_W} and
	\eqref{eq:rho} we obtain
	\begin{align*}
          c_{W,V}c_{V,W}(x_2v_0\otimes w_1)=&\;c_{W,V}(zw_1\otimes x_2v_0)\\
          =&\;\sigma(z)x_1x_2v_0\otimes w_1\\
          =&\;\sigma(z)(-\alpha^2\beta^2v_0+\alpha\beta x_2v_0)\otimes w_1.
	\end{align*}
	Since $\varphi _1=\id -c_{W,V}c_{V,W}$, this implies Equation~\eqref{eq:deg2:phi1(x2v0,w1)}.
\end{proof}

\begin{lem}
	\label{lem:deg2:X1}
	Assume that $\sigma(x_1)=-1$. Then 
    $X_1^{V,W}$ is absolutely simple if and only if $(1-\sigma (z)\alpha
    )(1-\sigma (z)\alpha \beta )=0$. In this case,
    $X_1^{V,W}\simeq M(x_1z,\sigma_1)$, where $\sigma_1$ is an
	absolutely irreducible representation of $G^{x_1z}=G^{x_1}$ with
	\begin{align*}
		\sigma_1(x_1)&=-\alpha^2\beta\sigma(z), 
		&\sigma_1(x_1x_2x_3)&=\epsilon\alpha^3, 
		&\sigma_1(z)=\sigma(z)\rho(z).
	\end{align*}
\end{lem}

\begin{proof}
  Since $\supp X_1^{V,W}=(x_1z)^G$ and the centralizer $G^{x_1z}=G^{x_1}$
  is abelian, $X_1^{V,W}$ is absolutely
  simple if and only if $\dim (X_1^{V,W})_{x_1z}=1$. Recall that
  $(X_1^{V,W})_{x_1z}=\Kspan\{\varphi _1(v_0\otimes w_1),\varphi
  _1(x_2v_0\otimes w_1)\}$. Thus Lemma~\ref{lem:deg2:X1_auxiliar} implies that
  $X_1^{V,W}$ is absolutely simple if and only if $(1-\sigma (z)\alpha )
  (1-\sigma (z)\alpha \beta )=0$.

	Let $w_1'=\varphi_1(x_2v_0\otimes w_1)$.
  Using Equation \eqref{eq:deg2:phi1(x2v0,w1)} and \eqref{eq:rho} we
	compute
	\begin{equation*}
		\label{eq:x1w1'}
			x_1 w_1' = (\alpha^2\beta^2 - \alpha^3\sigma(z) -
      \alpha^3\beta^2\sigma(z)) v_0\otimes w_1
			-\alpha\beta(1-\alpha\beta\sigma(z))x_2v_0\otimes w_1.
	\end{equation*}
  If $\sigma (z)\alpha =1$ then
  $$x_1w_1'=-\alpha^2v_0\otimes w_1-\alpha\beta(1-\beta)x_2v_0\otimes w_1
  =-\alpha\beta w_1',$$
  and if $\sigma (z)\alpha \beta =1$ then $x_1w_1'=-\alpha^2\beta v_0\otimes
  w_1=-\alpha w_1'$. In both cases we conclude that
  $x_1w_1'=-\alpha^2\beta\sigma(z)w_1'$. Since $z,x_1x_2x_3\in Z(G)$,
  $\sigma (x_1x_2x_3)=-\epsilon $, and $\rho (x_1x_2x_3)=-\alpha ^3$
  by Lemma~\ref{lem:deg2}, $\sigma _1$ has the claimed properties.
\end{proof}

\begin{lem}
	\label{lem:deg2:X2}
	Assume that $\sigma(x_1)=-1$ and $(1-\sigma (z)\alpha )(1-\sigma
  (z)\alpha \beta )=0$. Then $X_2^{V,W}=0$ if and only if $\rho(z)=-1$.
\end{lem}

\begin{proof}
  {}From \cite[Lemma 1.7]{MR2732989} we conclude that
  $$X_2^{V,W} = \K G\{\varphi_2(v_0\otimes w_1'),\varphi_2(x_2v_0\otimes
    w_1')\}.$$
    Since $x_1w_1'=-\alpha ^2\beta \sigma (z)w_1'$ by
    Lemma~\ref{lem:deg2:X1} and $x_1x_2v_0=\alpha \beta
    x_2v_0-\alpha^2\beta^2v_0$, the vanishing of $X_2^{V,W}$ is equivalent to
    the vanishing of $\varphi _2(x_2v_0\otimes w_1')$.
    
	We first compute
	\begin{equation}
		\label{eq:c^2(x2v0,w1')}
		\begin{aligned}
			c_{X_1,V}c_{V,X_1}(x_2v_0\otimes w_1')
			&=\rho(z)\sigma(z)c_{X_1,V}(w_1'\otimes x_2v_0)\\
			&=\rho(z)^2\sigma(z)(-\alpha^2\beta^2 v_0+\alpha\beta x_2v_0)\otimes w_1'.
		\end{aligned}
	\end{equation}
	Assume first that $\sigma(z)\alpha\beta=1$. Then
  $w_1'=\varphi_1(x_2v_0\otimes w_1)=\alpha\beta v_0\otimes w_1$.
	Since $\varphi_2=\id-c_{X_1,V}c_{V,X_1}+(\id\otimes\varphi_1)c_{1,2}$,
  Equation \eqref{eq:c^2(x2v0,w1')} implies that
	\[
		\varphi_2(x_2v_0\otimes w_1')=\alpha\beta\rho(z)(1+\rho(z))v_0\otimes
    w_1'+(1-\rho(z)^2)x_2v_0\otimes w_1'.
	\]

  Assume now that $\sigma(z)\alpha=1$. Then
  $w_1'=(1-\beta)x_2v_0\otimes w_1+\alpha\beta^2 v_0\otimes w_1$
  by Lemma \ref{lem:deg2:X1_auxiliar}.
  Using Equation \eqref{eq:c^2(x2v0,w1')}
  one obtains that
	\[
		\varphi_2(x_2v_0\otimes w_1')=(1+\rho(z))(\alpha\beta^2\rho(z)v_0\otimes
    w_1'+(1-\beta\rho(z))x_2v_0\otimes w_1').
	\]
  In both cases, $\varphi_2(x_2v_0\otimes w_1')=0$
  if and only if $\rho(z)=-1$.
\end{proof}

\begin{proof}[Proof of Proposition \ref{pro:deg=1}]
	Since $(V,W)$ admits all reflections and $\mathcal{W}(V,W)$ is finite,
	$(\ad W)^m(V)$ is absolutely simple or zero for all $m\in\N_0$ by
	\cite[Thm.\,7.2(3)]{MR2734956}.  The Cartan matrix of $(V,W)$ is
	non-diagonal and hence $(\ad W)(V)\ne0$. Then Lemma~\ref{lem:sigma(x1)=-1}
	implies that $a^{(V,W)}_{2,1}\leq-3$ and $\sigma (x_1)=-1$. Therefore
	$a^{(V,W)}_{1,2}=-1$ and $a^{(V,W)}_{2,1}=-3$ by assumption.
	Hence $X_m^{V,W}=0$ if and only if $m\ge 2$ and
	$X_m^{W,V}=0$ if and only if $m\ge 4$ by the definition of the entries of the
    Cartan matrix $A^{(V,W)}$.
	Further, $\deg\rho\leq2$
	by Lemma~\ref{lem:degrees}. 

    Suppose that $\deg\rho=2$.
    Then
    \begin{equation} 
        \label{eq:deg2}
        (\alpha\sigma(z)-1)(\alpha\beta\sigma(z)-1)=0 
    \end{equation} 
    by Lemma~\ref{lem:deg2:X1}, and $\rho(z)=-1$ by Lemma \ref{lem:deg2:X2}.
    From
    $a_{1,2}^{(V,W)}=-1$ we obtain that  $R_1(V,W)=\left(V^*,X_1^{V,W}\right)$.
    Since $\supp X_1^{V,W}=(x_1z)^G\simeq\chi_1^T$ and $\supp V^*\simeq\supp V$
    as quandles,
    Lemma~\ref{lem:sigma(x1)=-1}
    implies that $(\ad X_1^{V,W})^3(V^*)$ is non-zero.
    Then $(\ad X_1^{V,W})^3(V^*)$ is
    absolutely simple by \cite[Thm.\,7.2(3)]{MR2734956}.
    Now Lemma~\ref{lem:sigma(x1)=-1} implies that $\sigma_1(x_1z)=-1$.
    By Lemma~\ref{lem:deg2:X1}, 
    \[
        -1=\sigma_1(x_1z)=-\alpha^2\beta\sigma(z)^2\rho(z)=\alpha^2\beta\sigma(z)^2.
    \]
    On the other hand $\alpha^2\beta\sigma(z)^2\in\{\beta,\beta^2\}$ by
    Equation \eqref{eq:deg2}. This contradicts to $1+\beta+\beta^2=0$. Therefore $\deg\rho=1$. 
\end{proof}

\section{Proof of Theorem \ref{thm:T}}
\label{section:T}

As in Subsection~\ref{subsection:T:preliminaries}, let $V=M(z,\rho)$ and
$W=M(x_1,\sigma)$, and assume that $\deg\rho =1$.
Then
$\rho(x_1)=\rho(x_2)=\rho(x_3)=\rho(x_4)$ since $x_1,x_2,x_3,x_4$ are conjugate
elements.
We write $X_n=X_n^{V,W}$ and
$\varphi_n=\varphi_n^{V,W}$ for all $n\in\N_0$ if no confusion can arise. 

\begin{lem}
	\label{lem:X1}
    Assume that $\sigma(x_1)=-1$. Then 
    $X_1^{V,W}$ is  non-zero if and only if 
	$\rho(x_1)\sigma(z)\ne1$. In this case, 
    $X_1^{V,W}$ is absolutely simple and
    $X_1^{V,W}\simeq M(x_1z,\sigma_1)$, where $\sigma_1$ is an
	absolutely irreducible representation of the centralizer
	$G^{x_1z}=G^{x_1}$ and 
	\begin{align*} 
	  \sigma_1(x_1)=&\;-\rho(x_1),&
	  \sigma_1(x_1x_2x_3)=&\;-\epsilon \rho(x_1)^3,&
	  \sigma_1(z)=&\;\rho(z)\sigma(z).
	\end{align*}
    For $i\in\{1,2,3,4\}$ let $w_i'\coloneqq v\otimes w_i$. Then
    $w_1',w_2',w_3',w_4'$ is a basis of $X_1^{V,W}$. The degrees of these
    basis vectors are $x_1z$, $x_2z$, $x_3z$ and $x_4z$, respectively.
\end{lem}

\begin{proof}
    By \cite[Lemma 1.7]{MR2732989}, $X_1^{V,W}=\varphi_1(V\otimes W)=\K
    G\varphi_1(v\otimes w_1)$.  Then 
	\begin{equation}
		\label{eq:phi1(v,w1)}
		\varphi_1(v\otimes w_1) = (\id-c_{W,V}c_{V,W})(v\otimes w_1)
		=(1-\rho(x_1)\sigma(z))v\otimes w_1.
	\end{equation}
	Hence $v\otimes w_1\in(V\otimes W)_{x_1z}$ is non-zero if and only if
  $\rho (x_1)\sigma (z)\not=1$. Further,
	\begin{align*}
		x_1w_1' = x_1v\otimes x_1w_1 =-\rho(x_1)v\otimes w_1=-\rho(x_1)w_1'.
	\end{align*}
  The remaining claims on $\sigma _1$ follow from the absolute irreducibility
  of $V$ and $W$ and the facts that $x_1x_2x_3,z\in Z(G)$ and
  $X_1^{V,W}\subseteq V\otimes W$.
\end{proof}

\begin{rem}
	\label{rem:G_on_X1}
    To compute the action of $G$ on $X_1^{V,W}$ one has to note that
	\[
		x_i w_j'=x_i(v\otimes w_j)=x_iv\otimes x_iw_j=\rho(x_1)v\otimes x_iw_j
	\]
	and then use the action of $G$ on $W$ of Remark \ref{rem:G_on_W}.
\end{rem}

\begin{lem}
	\label{lem:X2}
    Assume that $\sigma(x_1)=-1$ and $\rho(x_1)\sigma(z)\ne1$. Then 
    $X_2^{V,W}=0$ if and only if $(1+\rho(z))(1-\rho(x_1z)\sigma(z))=0$.
\end{lem}

\begin{proof}
  Let $w_1'=v\otimes w_1$. Then
  $X_2^{V,W}=\varphi_2(V\otimes X_1) =\K G\varphi_2(v\otimes w_1')$
  by \cite[Lemma 1.7]{MR2732989}.
	Since $\varphi_2=\id-c_{X_1,V}c_{V,X_1}+(\id\otimes\varphi_1)c_{1,2}$,
	we first compute
	\[
		c_{X_1,V}c_{V,X_1}(v\otimes w_1')
		=c_{X_1,V}(zw_1'\otimes v)
		=x_1zv\otimes zw_1'
		=\rho(z)\rho(x_1z)\sigma(z) v\otimes w_1'.
	\]
	Now using Equation \eqref{eq:phi1(v,w1)} we compute 
	\begin{align*}
		(\id\otimes\varphi_1)c_{1,2}(v\otimes w_1') &=
    (\id\otimes\varphi_1)c_{1,2}(v\otimes v\otimes w_1)\\
		& =(\id\otimes\varphi_1)(\rho(z)v\otimes v\otimes w_1)\\
		& =\rho(z)(1-\rho(x_1)\sigma(z)) v\otimes v\otimes w_1.
	\end{align*}
	Hence $\varphi_2(v\otimes w_1')=(1+\rho(z))(1-\rho(x_1z)\sigma(z)) v\otimes
	w_1'$. This implies the claim.
\end{proof}

Now we compute the adjoint actions $(\ad W)^m(V)$ for $m\in\{2,3,4\}$. We
write $X_n=X_n^{W,V}$ and $\varphi_n=\varphi_n^{W,V}$ for all $n\in\N_0$ if no
confusion can arise.

\begin{rem}
	\label{rem:Y1}
	By \cite[Lemma 1.7]{MR2732989}, 
	\[
    X_1^{W,V}=\varphi_1(W\otimes V)=\K G\varphi_1(w_1\otimes v).
	\]
  Moreover, for all $i\in \{1,2,3,4\}$ we obtain that
	\begin{align}
		\label{eq:phi1(w1,v)}\varphi_1(w_i\otimes v)&=
    (1-\rho(x_1)\sigma(z))w_i\otimes v.
	\end{align}
\end{rem}

\begin{lem}
	\label{lem:Y1}
	Assume that $\rho(x_1)\sigma(z)\ne1$ and $\sigma(x_1)=-1$.  Then $X_1^{W,V}$ is
	absolutely simple.  Moreover, $X_1^{W,V}\simeq M(x_1z,\rho_1)$, where $\rho_1$ is
	an absolutely irreducible representation of $G^{x_1z}=G^{x_1}$ with
	\begin{align*}
		\rho_1(x_1)&=-\rho(x_1),&
		\rho_1(x_1x_2x_3)&=-\epsilon\rho(x_1)^3,&
		\rho_1(z)&=\rho(z)\sigma(z).
	\end{align*}
	For all $i\in \{1,2,3,4\}$ let $v_i'= w_i\otimes v$. Then
        $v_i'\in (X_1^{W,V})_{x_iz}$ for all $i$, and
        $v_1',v_2',v_3',v_4'$ form a basis of $X_1^{W,V}$. 
\end{lem}

\begin{proof}
  Since $c_{V,W}:X_1^{V,W}\to X_1^{W,V}$ is an isomorphism in $\ydG $,
  the claim follows from Lemma~\ref{lem:X1}.
\end{proof}

\begin{rem}
	\label{rem:G_on_Y1}
	Assume that $\rho(x_1)\sigma(z)\ne1$ and $\sigma(x_1)=-1$.
	For $j\in \{1,2,3,4\}$ let $v_j'=w_j\otimes v$.
        Remark~\ref{rem:G_on_W} implies that the action of $G$ on $X_1^{W,V}$
        is given by $zv_j'=\rho(z)\sigma(z)v_j'$ for all $j\in\{1,2,3,4\}$ and 
	\begin{align*}
		x_iv_j'=\begin{cases}
			-\rho(x_1) v_{i\triangleright j}' & \text{if $i=1$ or $j=1$ or $i=j$,}\\
			-\epsilon\rho(x_1) v_{i\triangleright j}' & \text{otherwise.}
		\end{cases}
	\end{align*}
\end{rem}

By \cite[Lemma 1.7]{MR2732989}, 
\[
X_2^{W,V}=\varphi_2(W\otimes X_1^{W,V})
    =\K G\{\varphi_2(w_1\otimes v_1'),\varphi_2(w_2\otimes v_1')\}.
\]
The two generators are computed in the following lemma.
	
\begin{lem}
	\label{lem:phi2}
	Assume that $\rho(x_1)\sigma(z)\ne1$ and $\sigma(x_1)=-1$. Then
	\begin{align}
		&\varphi_2(w_1\otimes v_1')=0,\\
		&\label{eq:phi2(w2,v1')}
		\begin{aligned}
			\varphi_2(w_2\otimes v_1')=w_2\otimes v_1'&-\epsilon\rho(x_1)\sigma(z) w_1\otimes v_3'\\
			&-(1-\rho(x_1)\sigma(z))w_3\otimes v_2'.
		\end{aligned}
	\end{align}
	Moreover, $v_1''\coloneqq x_2\varphi_2(w_2\otimes v_1')\in (W\otimes W\otimes
	V)_{x_2x_3z}$ is non-zero. 
\end{lem}

\begin{proof}
    We first prove that $\varphi_2(w_1\otimes v_1')=0$. Lemma~\ref{lem:Y1} implies that
	\begin{equation*}
		c_{X_1,W}c_{W,X_1}(w_1\otimes v_1') = c_{X_1,W}(x_1v_1'\otimes w_1)
		=-\rho_1(x_1)\sigma(z)w_1\otimes v_1'.
	\end{equation*}
	Then we compute
	\begin{align*}
		(\id\otimes\varphi_1)c_{1,2}(w_1\otimes v_1') &= (\id\otimes\varphi_1)c_{1,2}(w_1\otimes w_1\otimes v)\\
		&=(\id\otimes\varphi_1)(x_1w_1\otimes w_1\otimes v)\\
		&=-(1-\rho(x_1)\sigma(z))w_1\otimes v_1'
	\end{align*}
    using Equation \eqref{eq:phi1(w1,v)}.  Since $\rho_1(x_1)=-\rho(x_1)$ by
    Lemma \ref{lem:Y1}, we conclude that $\varphi_2(w_1\otimes v_1')=0$. Now we prove 
    Equation \eqref{eq:phi2(w2,v1')}. First we use Lemma
    \ref{lem:Y1} and Remark \ref{rem:G_on_Y1} to compute
	\[
		c_{X_1,W}c_{W,X_1}(w_2\otimes v_1')=x_3zw_2\otimes x_2v_1'=\epsilon\rho(x_1)\sigma(z) w_1\otimes v_3'.
	\]
	Then using Equation \eqref{eq:phi1(w1,v)} we obtain that
	\[
		(\id\otimes\varphi_1)c_{1,2}(w_2\otimes v_1')=-w_3\otimes (1-\rho(x_1)\sigma(z))v_2'. 
	\]
    These equations imply Equation \eqref{eq:phi2(w2,v1')}.  Now
    $\varphi_2(w_2\otimes v_1')\in(W\otimes W\otimes V)_{x_2x_1z}$ is non-zero,
    and hence $x_2\varphi_2(w_2\otimes v_1')\in (W\otimes W\otimes
    V)_{x_2x_3z}$ is non-zero.
\end{proof}

\begin{rem}
	\label{rem:v1''}
	Equation \eqref{eq:phi2(w2,v1')} and Remarks
	\ref{rem:G_on_W} and \ref{rem:G_on_Y1} imply that
	\begin{equation}
		\begin{aligned}
		\label{eq:v1''}
		v_1''=\rho(x_1)w_2\otimes v_3'&-\rho(x_1)^2\sigma(z) w_3\otimes v_4'\\
		&-\epsilon\rho(x_1)(1-\rho(x_1)\sigma(z))w_4\otimes v_2'.
		\end{aligned}
	\end{equation}
\end{rem}

\begin{lem}
  \label{lem:Y2}
  Assume that $\rho(x_1)\sigma(z)\ne1$ and $\sigma(x_1)=-1$. Then $X_2^{W,V}$ is
  absolutely simple if and only if
    \begin{gather}
      \label{eq:Y2:conditions}
      \epsilon=1,\quad
      (\rho (x_1)\sigma(z))^2-\rho(x_1)\sigma(z)+1=0.
    \end{gather}
    In this case, $X_2^{W,V}\simeq M(x_2x_3z,\rho_2)$, where $\rho_2$ is an
    absolutely irreducible representation of $G^{x_2x_3z}=G^{x_1}$ with 
    \begin{align*} 
      \rho_2(x_1)=-\rho(x_1)^2\sigma(z),&&
      \rho_2(x_1x_2x_3)=\rho(x_1)^3,&&
      \rho_2(z)=\rho(z)\sigma(z)^2. 
    \end{align*}
    Let $v_2''=\rho(x_1)\sigma (z)^2x_4v_1''$,
    $v_3''=\rho(x_1)\sigma (z)^2x_2v_1''$, and
    $v_4''=\rho(x_1)\sigma (z)^2x_3v_1''$.
    Then the set $\{v_1'',v_2'',v_3'',v_4''\}$
    is a basis of $X_2^{W,V}$.
    The degrees of these elements are $x_2x_3z$, $x_1x_4z$,
    $x_1x_2z$, and $x_1x_3z$, respectively.
\end{lem}

\begin{proof}
  By Remark \ref{rem:v1''}, Lemma \ref{lem:Y1} and Remark \ref{rem:G_on_Y1}, 
  \begin{align*}
    x_1v_1'' = \rho(x_1)^2w_4\otimes v_2'&-\rho(x_1)^3\sigma(z)w_2\otimes v_3'\\
    &-\epsilon\rho(x_1)^2(1-\rho(x_1)\sigma(z))w_3\otimes v_4'.
  \end{align*}
  Assume that $X_2^{W,V}$ is absolutely simple. Then $(X_2^{W,V})_{x_2x_3z}$ is
	$1$-dimensional, since the centralizer $G^{x_2x_3z}=G^{x_1}$
	is abelian. Hence
  $v_1''$ and $x_1v_1''$ are linearly dependent.  Relating the coefficients of
  $w_3\otimes v_4'$ and $w_4\otimes v_2'$ of $v_1''$
  and $x_1v_1''$, respectively, and using that $\epsilon ^2=1$, we conclude
  that $(\rho (x_1)\sigma(z))^2-\rho(x_1)\sigma(z)+1=0$. 
  Relating the coefficients of $w_3\otimes v_4'$ and $w_2\otimes v_3'$ of $v_1''$
  and $x_1v_1''$, respectively, we conclude that $\epsilon =1$.

  Conversely, \eqref{eq:Y2:conditions} implies that $x_1v_1''=-\rho
  (x_1)^2\sigma (z)v_1''$.  Since $x_1x_2x_3,z\in Z(G)$ and $z'u=\sigma
  (z')^2\rho (z')u$ for all $z'\in Z(G)$, $u\in W\otimes W\otimes V$, and since
  $G^{x_2x_3z}=\langle x_1,x_1x_2x_3,z\rangle $, the Yetter-Drinfeld module
  $X_2^{W,V}$ is absolutely simple if and only if \eqref{eq:Y2:conditions}
  holds. The above calculations also prove the formulas for $\rho _2$.  The
  last claim follows easily, since $x_2x_3z\in x_1^{-1}Z(G)$.
\end{proof}

\begin{rem}
	\label{rem:G_on_Y2}
  Assume that
	\[
		\sigma(x_1)=-1,\quad
		\epsilon =1,\quad
		(\rho (x_1)\sigma(z))^2-\rho(x_1)\sigma(z)+1=0.
	\]
  Let $v_j''$ for $j\in \{1,2,3,4\}$ be as in Lemma~\ref{lem:Y2}.
  The action of $G$ on $X_2^{W,V}$ is given by 
  \begin{align} \label{eq:G_on_Y2}
			zv_j''=\rho(z)\sigma(z)^2v_j'',\quad 
			x_iv''_j = -\rho(x_1)^2\sigma(z)v''_{i\triangleright j}
	\end{align}
	for all $i,j\in\{1,2,3,4\}$. Indeed,
  for $j=1$ this follows from the definition of $v''_j$ and from $\rho
  (x_1)^3\sigma (z)^3=-1$.
  Further, $v''_j=\rho (x_1)\sigma(z)^2x_{1\triangleright j}v''_1$
  for $j\in \{2,3,4\}$.
  Hence \eqref{eq:G_on_Y2} for $i=1$ and $j>1$ follows from
  $x_1x_{1\triangleright j}=x_{1\triangleright (1\triangleright j)}x_1$.
  For $i=j>1$, \eqref{eq:G_on_Y2} follows from
  $x_ix_{1\triangleright i}=x_{1\triangleright i}x_1$.
  For $i>1$, $j=1\triangleright i$, \eqref{eq:G_on_Y2} follows from
  $i\triangleright(1\triangleright i)=x_1$ and
  $x_ix_{1\triangleright j}=x_2x_3$ and from $\rho _2(x_2x_3)=-\rho
  (x_1)\sigma(z)^{-1}$.
  Finally, \eqref{eq:G_on_Y2} for $i>1$, $j=1\triangleright(1\triangleright i)$
  follows from $x_i^3v''_j=\rho _2(x_1^3)v''_j$, $i\triangleright
  (i\triangleright (i\triangleright j))=j$,
  and from the equations $(-\rho (x_1)^2\sigma(z))^3=\rho (x_1)^3=\rho _2(x_1^3)$.
\end{rem}

Recall that $v_1''\in (X_2^{W,V})_{x_2x_3z}$. By \cite[Lemma 1.7]{MR2732989}, 
\[
X_3^{W,V}=\varphi_3(W\otimes X_2^{W,V})=\K G\{\varphi_3(w_1\otimes v_1''),\varphi_3(w_2\otimes v_1'')\}.
\]
Therefore we need to compute $\varphi_3(w_1\otimes v_1'')$ and $\varphi_3(w_2\otimes
v_1'')$.

\begin{lem}
  \label{lem:phi2:auxiliar}
  Assume that $(\rho(x_1)\sigma(z))^2-\rho(x_1)\sigma(z)+1=0$,
  $\sigma(x_1)=-1$, and $\epsilon =1$.
  Then the following hold:
	\begin{align}
		\label{eq:phi2(w2,v2')}\varphi_2(w_2\otimes v_2')&=0,\\
		\label{eq:phi2(w2,v3')}\varphi_2(w_2\otimes v_3')&=\rho(x_1)^{-1}v_1'',\\
		\label{eq:phi2(w2,v4')}\varphi_2(w_2\otimes v_4')&=-\sigma(z)v_3'',\\
		\label{eq:phi2(w1,v3')}\varphi_2(w_1\otimes v_3')&=-\sigma(z)v_4'',\\
		\label{eq:phi2(w1,v2')}\varphi_2(w_1\otimes v_2')&=\rho(x_1)\sigma(z)^2v_3'',\\
    \label{eq:phi2(w1,v4')}\varphi_2(w_1\otimes v_4')&=\rho(x_1)^{-1}v_2''.
	\end{align}
\end{lem}

\begin{proof}
	By Lemma \ref{lem:phi2}, $\varphi_2(w_1\otimes v_1')=0$. Applying $x_4$ to
	this equation we obtain Equation \eqref{eq:phi2(w2,v2')}, where we used
  \eqref{eq:G_on_Y2}. To prove
	\eqref{eq:phi2(w2,v3')} we compute 
	\[
		v_1''=x_2\varphi_2(w_2\otimes v_1')=\varphi_2(x_2w_2\otimes x_2v_1')=\rho(x_1)\varphi_2(w_2\otimes v_3')
	\]
	using Remarks \ref{rem:G_on_W} and \ref{rem:G_on_Y1}, and Equation
	\eqref{eq:phi2(w2,v3')} follows. Now apply $x_2$ to Equation
	\eqref{eq:phi2(w2,v3')} to obtain Equation \eqref{eq:phi2(w2,v4')}.

	To prove Equation \eqref{eq:phi2(w1,v3')} apply $x_3$ to
	\eqref{eq:phi2(w2,v3')} and use Lemma \ref{lem:Y2} and
	\eqref{rem:G_on_Y2}.  Similarly, acting with $x_1$ on Equation
	\eqref{eq:phi2(w1,v3')} we obtain Equation \eqref{eq:phi2(w1,v2')}.  Finally,
	acting with $x_1$ on Equation \eqref{eq:phi2(w1,v2')} we obtain Equation
	\eqref{eq:phi2(w1,v4')}.
\end{proof}

\begin{lem}
  \label{lem:deg2:phi3(w2,v1'')}
  Assume that $(\rho(x_1)\sigma(z))^2-\rho(x_1)\sigma(z)+1=0$,
  $\sigma(x_1)=-1$, and $\epsilon =1$.
  Then $\varphi_3(w_2\otimes v_1'')=0$. 
\end{lem}

\begin{proof}
	Using Remark \ref{rem:G_on_W} and Lemma \ref{lem:Y2} we obtain that
	\begin{align*}
		c_{X_2,W}c_{W,X_2}(w_2\otimes v_1'')&=c_{X_2,W}(x_2v_1''\otimes w_2)
		=-\rho(x_1)^2\sigma(z)^2w_4\otimes v_3''.
	\end{align*}
	Using Equation \eqref{eq:v1''} we compute
	\begin{align*}
		(\id\otimes\varphi_2)c_{1,2}(w_2\otimes v_1'')=-&\rho(x_1) w_2\otimes\varphi_2(w_2\otimes v_3')\\
		&+\rho(x_1)^2\sigma(z) w_4\otimes \varphi_2(w_2\otimes v_4')\\
		&+\rho(x_1)(1-\rho(x_1)\sigma(z))w_1\otimes\varphi_2(w_2\otimes v_2').
	\end{align*}
  Since $\varphi _3=\id -c_{X_2,W}c_{W,X_2}+(\id \otimes \varphi _2)c_{1,2}$,
  Equations~\eqref{eq:phi2(w2,v2')}--\eqref{eq:phi2(w2,v4')}
  imply that
	\begin{align*}
		\varphi_3(w_2\otimes v_1'') = &\;w_2\otimes v_1''
    -\rho(x_1)^2\sigma(z)^2w_4\otimes v_3''\\
    &-\rho(x_1)w_2\otimes \rho (x_1)^{-1}v_1''
    +\rho(x_1)^2\sigma(z) w_4\otimes (-\sigma (z)v_3'').
	\end{align*}
	Thus the claim follows.
\end{proof}

\begin{lem}
	\label{lem:v1'''}
  Assume that $(\rho (x_1)\sigma(z))^2-\rho(x_1)\sigma(z)+1=0$,
	$\sigma(x_1)=-1$, and $\epsilon =1$.
	Let $v_1'''=\varphi_3(w_1\otimes v_1'')$.  Then
	\begin{equation}
		\label{eq:v1'''}
		\begin{aligned}
			v_1'''=\rho(x_1)\sigma(z)(w_1\otimes v_1''+w_2\otimes v_2''
			&+w_3\otimes v_3''+w_4\otimes v_4'')
		\end{aligned}
 	\end{equation}
	is a non-zero element of $(W\otimes W\otimes W\otimes V)_{x_1x_2x_3z}$.
\end{lem}

\begin{proof}
	Using Remark \ref{rem:G_on_W} and Lemma \ref{lem:Y2} we first compute
	\begin{align*}
		c_{X_2,W}c_{W,X_2}(w_1\otimes v_1'') &= c_{X_2,W}(x_1v_1''\otimes w_1)
		= -\rho(x_1)^2\sigma(z)^2 w_1\otimes v_1''.
  \end{align*}
	Using Equation \eqref{eq:v1''} and Remark \ref{rem:G_on_W}, a straightforward
	calculation yields
	\begin{align*}
		(\id\otimes\varphi_2)c_{1,2}(w_1\otimes v_1'') =& -\rho(x_1)w_4\otimes\varphi_2(w_1\otimes v_3')\\
		&+\rho(x_1)^2\sigma(z) w_2\otimes\varphi_2(w_1\otimes v_4')\\
		&+\rho(x_1)(1-\rho(x_1)\sigma(z))w_3\otimes\varphi_2(w_1\otimes v_2').
	\end{align*}
	Since $\varphi_3=\id-c_{X_2,W}c_{W,X_2}+(\id\otimes\varphi_2)c_{1,2}$,
	Equations \eqref{eq:phi2(w1,v3')}--\eqref{eq:phi2(w1,v4')}
  yield Equation~\eqref{eq:v1'''}.  The rest is clear.
\end{proof}

\begin{lem}
	\label{lem:Y3}
  Assume that $(\rho (x_1)\sigma(z))^2-\rho(x_1)\sigma(z)+1=0$,
	$\sigma(x_1)=-1$, and $\epsilon =1$.
  Then $X_3^{W,V}\simeq M(x_1x_2x_3z,\rho_3)$, where $\rho_3$ is an
  absolutely irre\-ducible representation of $G^{x_1x_2x_3z}=G$ with
	\begin{gather*}
        \rho_3(z)=\rho(z)\sigma(z)^3,\quad 
        \rho_3(x_1)=\rho_3(x_2)=\rho_3(x_3)=\rho_3(x_4)=\rho(x_1)^2\sigma(z).
	\end{gather*}
\end{lem}

\begin{proof}
  The formula for $\rho _3(z)$ follows from $zu=\rho (z)\sigma(z)^3u$ for all
  $u\in W^{\otimes 3}\otimes V$.
  By the remark above Lemma~\ref{lem:phi2:auxiliar}, by
  Lemma~\ref{lem:deg2:phi3(w2,v1'')}, and since
	$x_1x_2x_1^{-1} = x_4$ and $x_2x_1x_2^{-1}=x_3$,
  it is enough to show that
	$x_1v_1'''=x_2v_1'''=\rho (x_1)^2\sigma(z)v_1'''$.
	By Lemma \ref{lem:Y2}, $x_1v_1''=-\rho(x_1)^2\sigma(z)v_1''$. Then 
	\begin{align*}
		x_1v_1'''&= x_1\varphi_3(w_1\otimes v_1'')
    =\varphi_3(x_1w_1\otimes x_1v_1'')=\rho(x_1)^2\sigma(z)v_1'''.
	\end{align*}
  The claim on $x_2v_1'''$ follows from
	Equations~\eqref{eq:v1'''} and \eqref{eq:G_on_Y2}.
\end{proof}

By \cite[Lemma 1.7]{MR2732989} and since $x_iv_1'''=\rho_3(x_1)v_1'''$ for all
$i\in \{1,2,3,4\}$,
\[
X_4^{W,V}=\varphi_4(W\otimes X_3^{W,V})=\K G\varphi_4(w_1\otimes v_1''').
\]
The following lemma will be useful for computing $\varphi_4(w_1\otimes
v_1''')$. 

\begin{lem}
	\label{lem:phi3:auxiliar}
    Assume that $(\rho(x_1)\sigma(z))^2-\rho (x_1)\sigma(z)+1=0$,
    $\sigma(x_1)=-1$, and $\epsilon=1$. Then 
	\begin{align}
		\label{eq:phi3(w1,phi2(w1,*))}&\varphi_3(w_1\otimes v_2'')=\varphi_3(w_1\otimes v_3'')=\varphi_3(w_1\otimes v_4'')=0.
	\end{align}
\end{lem}

\begin{proof}
	By Lemma \ref{lem:deg2:phi3(w2,v1'')}, $\varphi_3(w_2\otimes v_1'')=0$.
  By acting on this equation with $x_3$ and using
	Equation~\eqref{eq:G_on_Y2} one obtains that
  $\varphi_3(w_1\otimes v_4'')=0$.
  The other two equations follow
	similarly by acting twice with $x_1$ on the latter equation.
\end{proof}

\begin{lem}
	\label{lem:Y4}
  Assume that $(\rho(x_1)\sigma(z))^2-\rho(x_1)\sigma(z)+1=0$,
  $\sigma(x_1)=-1$, and $\epsilon =1$.
  Then $X_4^{W,V}=0$. 
\end{lem}

\begin{proof}
	It is enough to prove that $\varphi_4(w_1\otimes v_1''')=0$.
	Lemma \ref{lem:Y3} implies that
	\[
		c_{X_3,W}c_{W,X_3}(w_1\otimes v_1''') 
		=c_{X_3,W}(x_1v_1'''\otimes w_1)
		=-\rho(x_1)^2\sigma(z)^2 w_1\otimes v_1'''.
	\]
	Equation \eqref{eq:v1'''} and Lemma \ref{lem:phi3:auxiliar} yield that
	\[
		(\id\otimes\varphi_3)c_{1,2}(w_1\otimes v_1''')
    =-\rho(x_1)\sigma(z)w_1\otimes v_1'''.
	\]
	Therefore $\varphi_4(w_1\otimes
	v_1''')=(\id-c_{X_3,W}c_{W,X_3}+(\id\otimes\varphi_3)c_{1,2})(w_1\otimes
	v_1''')=0$.  This proves the lemma.
\end{proof}

We summarize the results of this section in the following
proposition.

\begin{pro}
	\label{pro:cartan_matrix}
    Let $V=M(z,\rho)$ and $W=M(x_1,\sigma)$. Assume that
    $\deg\rho=\deg \sigma =1$ and that
    $\rho (x_1)\sigma (z)\ne1$, $\sigma(x_1)=-1$ and $\rho(x_1z)\sigma(z)=1$.
    Then the following hold:
	\begin{enumerate}
        \item $(\ad V)(W)$ is absolutely simple and $(\ad V)^2(W)=0$.
        \item The Yetter-Drinfeld modules
			$(\ad W)^m(V)$ are absolutely simple or zero for all
			$m\in\N_0$ if and only if
            $(\rho (x_1)\sigma (z))^2-\rho(x_1)\sigma (z)+1=0$ and 
		    $\epsilon =1$.
            In this case, $(\ad W)^3(V)\not=0$ and $(\ad W)^4(V)=0$.  
	\end{enumerate}
\end{pro}

\begin{proof}
	The first two claims follow from Lemmas \ref{lem:X1} and \ref{lem:X2}. The
	others are Lemmas \ref{lem:Y1}, \ref{lem:Y2}, \ref{lem:Y3} and \ref{lem:Y4}.
\end{proof}

Before proving Theorem~\ref{thm:T},
we need three more technical lemmas. Recall that if
$X$ is a finite-dimensional Yetter-Drinfeld module over $G$, then the dual
space $X^*$ is also a Yetter-Drinfeld module with
\[
    (gf)(x)=f(g^{-1}x),\quad
    f_{(-1)}\otimes f_{(0)}(y)=h^{-1}\otimes f(y)
\]
for all $g,h\in G$, $x\in X$, $y\in X_h$ and $f\in X^*$, where
$\delta(f)=f_{(-1)}\otimes f_{(0)}$. Further, if $X$ is simple then so is $X^*$.
In particular, $M(x,\gamma)^*\simeq M(x^{-1},\gamma^*)$ for all $x\in G$ and
all finite-dimensional representations $\rho$ of $G^x=G^{x^{-1}}$, where
$\gamma^*$ is the dual representation of $\gamma$. 

\begin{rem}
    \label{rem:T:R1}
    Let $x_1'\coloneqq x_1z$, $x_2'\coloneqq x_2z$, $x_3'\coloneqq x_3z$ and
    $x_4'\coloneqq x_4z$. Then $G=\langle x_1',x_2',x_3',x_4',z^{-1}\rangle$
    and the map $T\to G$,
    \[
        \chi_1\mapsto x_1',\quad
        \chi_2\mapsto x_2',\quad
        \chi_3\mapsto x_3',\quad
        \chi_4\mapsto x_4',\quad
        \zeta\mapsto z^{-1},
    \]
    is a group homomorphism. 
\end{rem}

\begin{lem}
    \label{lem:T:R1}
    Assume that
    $(\rho(x_1)\sigma(z))^2-\rho(x_1)\sigma(z)+1=0$, $\sigma(x_1)=-1$,
    $\sigma(x_2x_3)=1$, and $\rho(x_1z)\sigma(z)=1$. Then 
    \[
        R_1(V,W)=\left(V^*,X_1^{V,W}\right)
    \]
    with $V^*\simeq M(z^{-1},\rho^*)$, where $\rho^*$ is the irreducible
    representation of $G^z$ dual to $\rho$, $X_1^{V,W}\simeq M(x_1z,\sigma_1)$,
    where $\sigma_1$ is the irreducible representation of $G^{x_1z}$ given in
    Lemma \ref{lem:X1}, and 
    \begin{align}
        \label{eq:T:conditions:R1:1}&\sigma_1(x_1z)=-1,&
        &\sigma_1(x_2x_3z^2)=1,\\
        \label{eq:T:conditions:R1:2}&\rho^*(x_1)\sigma_1(z^{-1})=1,&
        &(\rho^*(x_1z)\sigma_1(z^{-1}))^2
        -\rho^*(x_1z)\sigma_1(z^{-1})+1=0.
    \end{align}
\end{lem}

\begin{proof}
    Since $\sigma(x_1)=-1$, $\rho(x_1)\sigma(z)\ne1$, and $\rho
    (x_1z)\sigma(z)=1$, the description of
    $R_1(V,W)$ follows from Proposition \ref{pro:cartan_matrix} and Lemma
    \ref{lem:X1}. Further, $\sigma_1(x_1)=-\rho(x_1)$,
    $\sigma_1(x_2x_3)=\sigma(x_2x_3)\rho(x_1)^2$ and
    $\sigma_1(z)=\rho(z)\sigma(z)$. Then $\sigma_1(x_1z)=-1$ and 
    \[
        \sigma_1(x_2x_3z^2)=\sigma(x_2x_3)\rho(x_1)^2\rho(z)^2\sigma(z)^2=1.
    \]
    Further,
    \begin{align*}
      \rho^*(x_1)\sigma_1(z^{-1})=&\rho(x_1)^{-1}\rho(z^{-1})\sigma(z^{-1})=1,
    \end{align*}
    which proves the first equation in \eqref{eq:T:conditions:R1:2}.
    Since
    \begin{align*}
      \rho^*(x_1z)\sigma_1(z^{-1})=&\rho(x_1z)^{-1}\rho(z)^{-1}\sigma(z)^{-1}=
      \rho (x_1)\sigma (z),
    \end{align*}
    the second equation in \eqref{eq:T:conditions:R1:2} also holds.
\end{proof}

\begin{rem}
    \label{rem:T:R2}
    Let $x_1''= x_1^{-1}$, $x_2''=x_2^{-1}$, $x_3''=x_4^{-1}$,
    $x_4''=x_3^{-1}$, and $z''=x_1x_2x_3z$. Then $G=\langle
    x_1'',x_2'',x_3'',x_4'',z''\rangle$ and the map $T\to G$,
    \[
        \chi_1\mapsto x_1'',\quad
        \chi_2\mapsto x_2'',\quad
        \chi_3\mapsto x_3'',\quad
        \chi_4\mapsto x_4'',\quad
        \zeta\mapsto z'',
    \]
    is a group homomorphism. 
\end{rem}

\begin{lem}
    \label{lem:T:R2}
    Assume that $(\rho(x_1)\sigma(z))^2-\rho(x_1)\sigma(z)+1=0$, $\sigma(x_1)=-1$,
    $\sigma(x_2x_3)=1$, and $\rho(x_1z)\sigma(z)=1$. Let
    $x_1''$, $x_2''$, $x_3''$, $x_4''$ and $z''$ be as in Remark~\ref{rem:T:R2}.
    Then 
    \[
        R_2(V,W)=\left(X_3^{W,V},W^*\right)
    \]
    with $X_3^{W,V}\simeq M(z'',\rho_3)$, where 
    $\rho_3$ is the irreducible representation of $G$ given in
    Lemma \ref{lem:Y3}, $W^*\simeq M(x_1'',\sigma^{*})$, where $\sigma^*$ is
    the irreducible representation of $G^{x_1}$ dual to $\sigma$, and 
    \begin{align}
        \label{eq:T:conditions:R2:1}
        &\sigma^*(x_1'')=-1, & &\sigma^*(x_2''x_3'')=1,\\
        \label{eq:T:conditions:R2:2}
        & \rho_3(x_1''z'')\sigma^*(z'')=1,&
        &(\rho_3(x_1'')\sigma^*(z''))^2
        -\rho_3(x_1'')\sigma^*(z'')+1=0.
    \end{align}
\end{lem}

\begin{proof}
    The description of $R_2(V,W)$ follows from
    Proposition \ref{pro:cartan_matrix} and Lemma \ref{lem:Y3}.
    Equation~\eqref{eq:T:conditions:R2:1} follows from the formulas
    \begin{gather*}
      \sigma^*(x_1^{-1})=\sigma(x_1)=-1,\\
      \sigma^*(x_2^{-1}x_4^{-1})=\sigma^*((x_4x_2)^{-1})=\sigma(x_2x_3)=1.
    \end{gather*}
    Similarly, \eqref{eq:T:conditions:R2:2} follows from the calculations
    \begin{align*}
      &\rho_3(x_2x_3z)\sigma^*(x_1x_2x_3z)=
      \rho(x_1)^4\sigma(z)^5\rho(z)(-\sigma (z))^{-1}
      =-\rho(x_1)^3\sigma(z)^3=1,\\
      &\rho_3(x_1^{-1})\sigma^*(x_1x_2x_3z)=
      \rho(x_1)^{-2}\sigma(z)^{-1}(-\sigma (z))^{-1}
      =\rho(x_1)\sigma(z),
    \end{align*}
    where the last equation is valid because of $(\rho (x_1)\sigma(z))^3=-1$.
\end{proof}

\begin{lem}
    \label{lem:T:braided_isomorphisms}
    Assume that $(\rho(x_1)\sigma(z))^2-\rho(x_1)\sigma(z)+1=0$, $\sigma(x_1)=-1$,
    $\epsilon =1$, and $\rho(x_1z)\sigma(z)=1$.
    Then
    \begin{align}
        \label{eq:V}&V\simeq M(z^{-1},\rho^*)\simeq M(x_1x_2x_3z,\rho_3),\\
        \label{eq:W}&W\simeq M(x_1z,\sigma_1)\simeq M(x_1^{-1},\sigma^*)
    \end{align}
    as braided vector spaces.
\end{lem}

\begin{proof}
   Let $f:V\to M(z^{-1},\rho^*)$ be a non-zero linear map.
   Then
   \begin{align*}
       (f\otimes f)c_{V,V}(v\otimes v)=(f\otimes f)(zv\otimes v)
       =\rho(z)f(v)\otimes f(v).
   \end{align*}
   On the other hand, 
   \begin{align*}
       c_{V^*,V^*}(f\otimes f)(v\otimes v)=z^{-1}f(v)\otimes f(v)=\rho
       ^*(z^{-1})f(v)\otimes f(v).
   \end{align*}
   Since $\rho ^*(z^{-1})=\rho(z)$, we conclude that $V$ and
   $M(z^{-1},\rho^*)$ are
   isomorphic as braided vector spaces. Similarly, $V$ and
   $M(x_1x_2x_3z,\rho_3)$ are isomorphic as braided vector spaces,
   since $\rho_3(x_1x_2x_3z)=\rho(z)$. Indeed,
   $$\rho _3(x_1x_2x_3z)=(\rho(x_1)^2\sigma(z))^3\rho(z)\sigma(z)^3
     =(\rho(x_1)\sigma(z))^6\rho(z)=\rho(z)
   $$
   by Lemma~\ref{lem:Y3} and since $(\rho (x_1)\sigma(z))^6=1$.

   We now prove that $W\simeq X_1^{V,W}$ as braided vector spaces.
   Then $W$ and $M(x_1z,\sigma_1)$ are isomorphic as braided vector spaces by
   Lemma~\ref{lem:X1}. Let  
   \[
   f\in \mathrm{Hom}\left(W,X_1^{V,W}\right),\quad
   w_i\mapsto w_i'\text{ for $i\in\{1,2,3,4\}$,}
   \]
   where $w_i'=v\otimes w_i$ for all $i$.
   Then
   \begin{align*}
       (f\otimes f)&c_{W,W}(w_i\otimes w_j)
       =(f\otimes f)(x_iw_j\otimes w_i)=-f(w_{i\triangleright j})\otimes
       f(w_i),
   \end{align*}
   and on the other hand 
	\begin{align*}
		c_{X_1,X_1}&(f\otimes f)(w_i\otimes w_j)
		=c_{X_1,X_1}(w_i'\otimes w_j')=x_izw_j'\otimes w_i'.
	\end{align*}
  Since $x_iw_j=-w_{i\triangleright j}$ and
  $x_izw_j'=-\rho (x_1z)\sigma(z)w'_{i\triangleright j}
  =-w'_{i\triangleright j}$ by Remark~\ref{rem:G_on_W},
  we conclude that $W\simeq X_1^{V,W}$ as braided vector spaces.

  Similarly, $W$ and $M(x_1^{-1},\sigma^*)$
  are isomorphic as braided vector spaces.
  Indeed, let $x_1''$, $x_2''$, $x_3''$, and $x_4''$ be as in Remark~\ref{rem:T:R2}.
  Then $\sigma^*(x_1'')=-1$ and $\sigma^*(x_2''x_3'')=1$
  by Lemma~\ref{lem:T:R2}, and hence by Remark~\ref{rem:G_on_W} there is a
  basis $w_1'',w_2'',w_3'',w_4''$ of $M(x_1^{-1},\sigma^*)$
  such that $x_i''w_j''=-w_{i\triangleright j}''$ for all $i,j\in
  \{1,2,3,4\}$. This implies that $W\simeq M(x_1^{-1},\sigma^*)$ as braided
  vector spaces.
\end{proof}

\begin{proof}[Proof of Theorem \ref{thm:T}]
    $(1)\implies(2)$.  Since $\NA(V\oplus W)<\infty$, the pair $(V,W)$ admits
    all reflections by \cite[Cor 3.18]{MR2766176} and the Weyl groupoid is
    finite by \cite[Prop.\,3.23]{MR2766176}. 

    $(2)\implies(3)$. By \cite[Prop.\,4.3]{partII}, after changing the object of
    $\mathcal{W}(V,W)$ and possibly interchanging $V$ and $W$, we may assume
    that $V=M(z,\rho)$ and $W=M(x_1,\sigma)$ satisfy $(\ad V)(W)\ne0$, $(\ad
    V)^2(W)=0$ and $(\ad W)^4(V)=0$. Further, $\deg\rho=1$ by Proposition~\ref{pro:deg=1}.
    By \cite[Thm.\,7.2(3)]{MR2734956}, $(\ad W)^m(V)$ is
    absolutely simple or zero for all $m\in\N_0$.
    Lemma~\ref{lem:sigma(x1)=-1}
    implies that $(\ad W)^2(V)\ne0$ and $(\ad W)^3(V)\ne0$. Hence
    $\sigma(x_1)=-1$ by Lemma~\ref{lem:sigma(x1)=-1}. Since $(\ad V)(W)$ is
    non-zero, we obtain from Lemma \ref{lem:X1} that $\rho(x_1)\sigma(z)\ne1$.
    Further,
    $(\ad W)^2(V)$ is absolutely simple, and hence $\sigma(x_2x_3)=1$
    and
    $(\rho(x_1)\sigma(z))^2-\rho(x_1)\sigma(z)+1=0$ by Lemma \ref{lem:Y2}.
    Since $(\ad V)^2(W)=0$, we obtain that
    $R_1(V,W)=\left(V^*,X_1^{V,W}\right)$.
    Now $\supp X_1^{V,W}\simeq(x_1z)^G\simeq\chi_1^T$ and $\supp V^*\simeq
    \supp V$
    imply that $(\ad X_1^{V,W})^3(V^*)$ is absolutely simple or zero by
    \cite[Thm.\,7.2(3)]{MR2734956}. Hence $\sigma_1(x_1z)=-1$ by
    Lemma~\ref{lem:sigma(x1)=-1} and therefore
    $\rho(x_1z)\sigma(z)=1$ by Lemma~\ref{lem:X1}. 

    $(3)\implies(1)$.  By Proposition \ref{pro:cartan_matrix} and Lemmas
    \ref{lem:T:R1} and \ref{lem:T:R2} the Weyl groupoid of $(V,W)$ is standard
    with Cartan matrix of type $G_2$.  Suppose that the Cartan matrix of
    $(V,W)$ is $A^{(V,W)}=
		\begin{pmatrix}
			2 & -1\\
			-3 & 2
		\end{pmatrix}
	$. 
    Then $s_2s_1s_2s_1s_2s_1$ is a reduced decomposition of the longest word in
    the Weyl group of $A^{(V,W)}$. With respect to this reduced decomposition one
    obtains 
	\begin{align*}
		&\beta_1=\alpha_2, && \beta_2=3\alpha_2+\alpha_1,\\
		&\beta_3=2\alpha_2+\alpha_1, &&\beta_4=3\alpha_2+2\alpha_1,\\
		&\beta_5=\alpha_2+\alpha_1, &&\beta_6=\alpha_1,
	\end{align*}
    where $\{\alpha_1,\alpha_2\}$ is the standard basis of $\Z^2$. 
    Since $A^{(V,W)}$ is of finite Cartan type, the set of real roots
    associated to the pair $(V,W)$ is finite by \cite[Cor.  2.4]{MR2732989}.
    By \cite[Thm.\,2.6]{MR2732989}, 
    \[
        \NA(V\oplus W)\simeq \NA(M_{\beta_6})\otimes
        \NA(M_{\beta_5})\otimes
        \cdots
        \otimes
        \NA(M_{\beta_2})\otimes
        \NA(M_{\beta_1}),
    \]
    as $\N_0^2$-graded vector spaces, where $\deg M_{\beta_j}=\beta_j$ for all
		$j\in\{1,\dots,6\}$, $M_{\beta _1}=W$, $M_{\beta _6}=V$, and 
		$M_{\beta_2},M_{\beta_3},M_{\beta _4},M_{\beta_5}\subseteq\NA(V\oplus W)$ are
    certain finite-dimensional simple subobjects in $\ydG$. Moreover,
    Lemmas \ref{lem:T:R1}, \ref{lem:T:R2} and \ref{lem:T:braided_isomorphisms}
    imply that
	\begin{gather*}
        M_{\beta_1}\simeq M_{\beta_3}\simeq M_{\beta_5}\simeq W\text{ and }
        M_{\beta_2}\simeq M_{\beta_4}\simeq M_{\beta_6}\simeq V
	\end{gather*}
    as braided vector spaces.
		Indeed, by Lemma~\ref{lem:T:R1} we can apply our theory to
		the pair $R_1(V,W)=(V^*,X_1^{V,W})$ if we replace $z,\rho,x_1$, and $\sigma
		$ by $z^{-1},\rho^*,x_1z$, and $\sigma _1$, respectively, and also to the
		pair $R_2(V,W)$ by Lemma~\ref{lem:T:R2}. Since
		$s_1s_2(\alpha_1)=3\alpha_2+2\alpha_1$, we conclude
    from \cite[Thm.\,2.6]{MR2732989}(1), that $M_{\beta _4}$ is
		isomorphic to the first entry of $R_1R_2(V,W)$ in $\ydG$, and hence to $V$
		as a braided vector space by iterated application of
		Lemma~\ref{lem:T:braided_isomorphisms}. The other isomorphisms follow similarly.
		Therefore the Nichols algebras of the braided
		vector spaces $M_{\beta _k}$, $1\le k\le 6$,
		are finite-dimensional with Hilbert series 
    \[
    \mathcal{H}_{\NA(M_{\beta_k})}(t)=
	\begin{cases}
		(6)_t & \text{if $\mathrm{char}\,\K\ne2$,}\\
		(3)_t & \text{if $\mathrm{char}\,\K=2$},
	\end{cases}
    \]
    for all $k\in\{2,4,6\}$, see \cite[Section 3]{MR2207786}, and 
    \[
    \mathcal{H}_{\NA(M_{\beta_l})}(t)=
	\begin{cases}
		(2)^2_t(3)_t(6)_t & \text{if $\mathrm{char}\,\K\ne2$,}\\
		(2)^2_t(3)_t^2 & \text{if $\mathrm{char}\,\K=2$},
	\end{cases}
	\]
    for all $l\in\{1,3,5\}$, see \cite[Thm.\,6.15]{MR1994219} and
    \cite[Prop.\,5.6]{MR2803792}. From this the claim follows.
\end{proof}

\section{Nichols algebras over epimorphic images of $\Gamma_4$}
\label{section:basics:G4}

\subsection{Preliminaries}
\label{subsection:G4:preliminaries}

Recall from \cite[Section 3]{MR2732989} that the group $\Gamma_n$ for $n\geq2$ is
isomorphic to the group given by generators $a$, $b$, $\nu $ and relations
\begin{align*}
	&ba=\nu ab,\quad
    a\nu=\nu^{-1}a,\quad
    b\nu=\nu b,\quad
    \nu^n=1.
\end{align*}
The case $n=2$ was studied in \cite{MR2732989},
and the case $n=3$ appears to be more complicated.
Here we concentrate on the case where $n=4$.  By
\cite[Section 3]{MR2732989}, the center of $\Gamma_4$ is
$Z(\Gamma_4)=\langle\nu^{-1}b^2,b^4,a^2\rangle$. 

In what follows, let $G$ be a group and let $g,h,\epsilon\in G$. Assume that
$G=\langle g,h,\epsilon\rangle$, $\epsilon^2\ne1$, and that there is a group
homomorphism $\Gamma_4\to G$ with $a\mapsto g$, $b\mapsto h$ and
$\nu\mapsto\epsilon$. Then $G$ is a non-abelian quotient of $\Gamma_4$ such
that $|g^G|=4$ and $|h^G|=2$. Further, $\epsilon^{-1}h^2,h^4,g^2\in Z(G)$.

Let $V=M(h,\rho)$, where $\rho$ is an absolutely irreducible representation of
the centralizer $G^h=\langle \epsilon,h,g^2\rangle=\langle h\rangle Z(G)$.
Then $\deg\rho =1$ since $G^h$ is abelian. Let $v\in V_h$ with $v\ne0$. The
elements $v$, $gv$ form a basis of $V$. The degrees of these basis vectors are
$h$ and $ghg^{-1}=\epsilon^{-1}h$, respectively. The support of $V$ is
isomorphic to the trivial quandle with two elements.

\begin{rem}
	\label{rem:G4:G_on_V}
	Assume that $\rho(h)=-1$. Then the action of $G$ on $V$ is given by the
	following table:
	\begin{center}
		\begin{tabular}{c|cc}
			$V$ & $v$ & $gv$ \tabularnewline
			\hline
			$\epsilon$ & $\rho(\epsilon)v$ & $\rho(\epsilon)^{-1}gv$ \tabularnewline
			$h$ & $-v$ & $-\rho(\epsilon)^{-1}gv$ \tabularnewline
			$g$ & $gv$ & $\rho(g^2)v$ \tabularnewline
		\end{tabular}
	\end{center}
\end{rem}

Let $W=M(g,\sigma)$, where $\sigma$ is an absolutely irreducible representation
of the centralizer $G^g=\langle \epsilon^2,\epsilon^{-1}h^2,g\rangle=\langle
g\rangle Z(G)$. Then $\deg \sigma=1$ since $G^g$ is abelian. Let $w\in W_g$
with $w\ne0$. The elements $w$, $hw$, $\epsilon w$, $\epsilon hw$ form a basis
of $W$. The degrees of these basis vectors are $g$, $\epsilon g$, $\epsilon^2g$
and $\epsilon^3 g$, respectively. The support of $W$ is isomorphic to the
dihedral quandle with four elements.

\begin{rem}
  \label{rem:G4:G_on_W}
  Assume that $\sigma(g)=-1$. Then the action of $G$ on $W$ is given by the
  following table:
  \begin{center}
    \begin{tabular}{c|cccc}
      $W$ & $w$ & $hw$ & $\epsilon w$ & $\epsilon hw$ \tabularnewline
      \hline
      $\epsilon$ & $\epsilon w$ & $\epsilon hw$ & $\sigma(\epsilon^2)w$ &
      $\sigma(\epsilon^2)hw$ \tabularnewline
      $h$ & $hw$ & $\sigma(\epsilon^{-1}h^2)\epsilon w$ & $\epsilon hw$ &
      $\sigma(\epsilon^2)\sigma(\epsilon^{-1}h^2)w$ \tabularnewline
      $g$ & $-w$ & $-\sigma(\epsilon^2)\epsilon hw$ &
      $-\sigma(\epsilon^2)\epsilon w$ & $-\sigma(\epsilon^2)hw$
      \tabularnewline
    \end{tabular}
  \end{center}
\end{rem}

\begin{rem}
	\label{rem:amalgamated}
	Let us describe the quandle structure of $\supp(V\oplus W)$.  Of course, the
	quandle $\supp(V\oplus W)$ is isomorphic to the conjugation quandle
	$h^{\Gamma_4}\cup g^{\Gamma_4}$. An alternative description for this quandle
	goes as follows:
	
	As we said before, $\supp V$ is a trivial quandle with two elements and
	$\supp W$ is a dihedral quandle with four elements. Thus we may assume that
	the quandle $\supp V$ is isomorphic to the quandle $Y=\{y_1,y_2\}$ given by
	$y_i\triangleright y_j=y_j$ for all $i,j\in\{1,2\}$ and that $\supp W$ is the
	quandle over $Z=\{z_1,z_2,z_3,z_4\}$ given by $z_i\triangleright
	z_j=z_{2i-j\bmod{4}}$ for all $i,j\in\{1,2,3,4\}$. The quandle $\supp(V\oplus
	W)$ is then isomorphic to the amalgamated sum of $Y$ and $Z$ with respect to
	the morphisms $\sigma(y)=(z_1\;z_2)$ for all $y\in Y$ and 
	\[
	\tau(z)=\begin{cases}
			(y_1\;y_2\;y_3\;y_4)&\text{if $z=z_1$},\\
			(y_1\;y_4\;y_3\;y_2)&\text{if $z=z_2$}.
		\end{cases}
	\]
	For the notion of amalgamated sum of quandles we refer to \cite[Lemma~1.18]{MR1994219}.
\end{rem}

\begin{lem}
	\label{lem:G4:rho(g)=-1}
		Let $V,W\in\ydG$ such that $\supp V=h^G$ and $\supp W=g^G$. Then the
		following hold:
		\begin{enumerate}
			\item $(\ad W)(V)$ and $(\ad W)^2(V)$ are non-zero. 
			\item If $(\ad V)^2(W)=0$ then $\dim V_h=1$ and $hv=-v$ for all $v\in V_h$.
			\item If $\supp(\ad W)^2(V)$ is a conjugacy class of $G$ then $\dim W_g=1$
				and $gw=-w$ for all $w\in W_{g}$.
		\end{enumerate}
\end{lem}

\begin{proof}
	Since $g$ and $h$ do not commute, $(\ad W)(V)$ is non-zero.
	Since $(g,h)\in\supp Q_1(g,h)$, \cite[Prop.\,5.5]{partII} yields that
        $(\epsilon^2g,\epsilon g,h)\in\supp Q_2(\epsilon g,g,h)$.  To prove (3) use
	Proposition \ref{pro:degrees} with $m=i=1$, $p_1=g$ and $p_2=h$.  Similarly,
	(2) follows from Proposition \ref{pro:degrees} with $m=i=1$, $p_1=h$ and
	$p_2=g$.
\end{proof}

\subsection{Main results}

Let $G$, $V$ and $W$ as in Subsection \ref{subsection:G4:preliminaries}.
Our goal is Theorem~\ref{thm:G4} below.  

\begin{thm}
	\label{thm:G4}
	Let $V=M(h,\rho)$ and $W=M(g,\sigma)$ be absolutely simple
	Yetter-Drinfeld modules over $G$. 
	Assume that $(\id -c_{W,V}c_{V,W})(V\otimes W)\not=0$. The following
	are equivalent:
    \begin{enumerate}
        \item The Nichols algebra $\NA (V\oplus W)$ is finite-dimensional.
        \item The pair $(V,W)$ admits all reflections and $\mathcal{W}(V,W)$ is
            finite.
        \item $\rho(h)=-1$, $\sigma (g)=-1$,
            $\rho(\epsilon)=\rho(g^2)\sigma(\epsilon^{-1}h^2)$ and
            $\rho(\epsilon)^2=-1$.
    \end{enumerate}
	In this case, $\mathcal{W}(V,W)$ is standard with Cartan
	matrix of type $B_2$.  Moreover, 
	\begin{align*}
		\mathcal{H}_{\NA (V\oplus W)}(t_1,t_2)=
    (1+t_2)^4(1+t_2^2)^2(1+t_1t_2)^4(1+t_1^2t_2^2)^2q(t_1t_2^2)q(t_1),
	\end{align*}
	where
	\begin{gather*}
		q(t)=\begin{cases}
			(1+t)^2(1+t^2) & \text{if $\mathrm{char}\,\K\ne2$},\\
			(1+t)^2 & \text{if $\mathrm{char}\,\K=2$}.
		\end{cases}
	\end{gather*}
	In particular, 
	\[
		\dim \NA (V\oplus W)=\begin{cases}
			8^264^2=262144 & \text{if $\mathrm{char}\,\K\ne2$},\\
   		4^264^2=65536 & \text{if $\mathrm{char}\,\K=2$}.
	\end{cases}
	\]
\end{thm}

We will prove Theorem \ref{thm:G4} in Section \ref{section:G4}. 

\section{Proof of Theorem \ref{thm:G4}}
\label{section:G4}

As in Subsection \ref{subsection:G4:preliminaries} let $V=M(h,\rho)$ and
$W=M(g,\sigma)$.
We assume that $\deg \rho =\deg \sigma =1$ and that $\sigma(g)=\rho(h)=-1$.  As usual,
$X_n=X_n^{V,W}$ and $\varphi_n=\varphi_n^{V,W}$ if no confusion can arise.  

\begin{lem}
	\label{lem:G4:X1}
	Assume that $\sigma(g)=\rho(h)=-1$. Then 
	$X_1^{V,W}$ is absolutely simple if and only if
    $\rho(\epsilon)=\rho(g^2)\sigma(\epsilon^{-1}h^2)$. In this case, $X_1^{V,W}\simeq
	M(hg,\sigma_1)$, where $\sigma_1$ is the irreducible character of the
	centralizer $G^{hg}=\langle\epsilon^2,\epsilon^{-1}h^2,hg\rangle$ given by
	\begin{gather*}
		\sigma_1(hg)=-1,\quad
		\sigma_1(\epsilon^2)=\rho(\epsilon^2)\sigma(\epsilon^2),\quad
		\sigma_1(\epsilon^{-1}h^2)=\sigma(\epsilon^{-1}h^2)\rho(\epsilon^{-1}h^2).
	\end{gather*}
	Let $w'\coloneqq\varphi_1(v\otimes w)$. Then $w'\in(V\otimes W)_{hg}$ is
	non-zero. Moreover, the set $\{w',hw',\epsilon w',\epsilon hw'\}$ is a
	basis of $X_1^{V,W}$. The degrees of these basis vectors are $hg$, $\epsilon hg$,
	$\epsilon^2 hg$ and $\epsilon^3 hg$, respectively.
\end{lem}

\begin{proof}
    First notice that $X_1^{V,W}=\varphi_1(V\otimes W)=\K G\varphi_1(v\otimes w)$. A
    direct calculation using Remarks \ref{rem:G4:G_on_V} and
    \ref{rem:G4:G_on_W} yields
    \begin{align*}
			w'=\varphi_1(v\otimes w)=(\id-c_{W,V}c_{V,W})(v\otimes w)=v\otimes w-\rho(\epsilon)^{-1}gv\otimes hw.
    \end{align*}
    Hence $w'\in(V\otimes W)_{hg}$ is non-zero. Since $G^{hg}=hgZ(G)$ is abelian,
    we conclude that $X_1^{V,W}$ is
    absolutely simple if and only if $hg w'\in \K w'$.  This is equivalent to
    $\rho(\epsilon)=\rho(g^2)\sigma(\epsilon^{-1}h^2)$, and then $hgw'=-w'$.
    The remaining claims on $\sigma _1$ follow from $\epsilon^2,\epsilon^{-1}h^2\in Z(G)$.
\end{proof}

\begin{rem}
	\label{rem:G4:G_on_X1}
	Assume that $\sigma(g)=\rho(h)=-1$ and
	$\rho(\epsilon)=\rho(g^2)\sigma(\epsilon^{-1}h^2)$. Then 
	action of $G$ on $X_1$ is given by the following table:
	\begin{center}
		\begin{tabular}{c|cccc}
			$X_1^{V,W}$ & $w'$ & $hw'$ & $\epsilon w'$ & $\epsilon hw'$ \tabularnewline
			\hline
			$\epsilon$ & $\epsilon w'$ & $\epsilon hw'$ & $\sigma_1(\epsilon^2)w'$ & $\sigma_1(\epsilon^2)hw'$ \tabularnewline
			$h$ & $hw'$ & $\sigma_1(\epsilon^{-1}h^2)\epsilon w'$ & $\epsilon hw'$ & $\sigma_1(\epsilon h^2)w'$ \tabularnewline
			$g$ & $-\sigma_1(\epsilon^{-1}h^{-2})\epsilon hw'$ & $-\sigma_1(\epsilon^2)\epsilon w'$ & $-\sigma_1(\epsilon^{-1}h^{-2})hw'$  & $-\sigma_1(\epsilon^2)w'$ \tabularnewline
		\end{tabular}
	\end{center}
\end{rem}

\begin{lem}
	\label{lem:G4:X2}
	Assume that $\sigma(g)=\rho(h)=-1$ and
	$\rho(\epsilon)=\rho(g^2)\sigma(\epsilon^{-1}h^2)$. Then $X_2^{V,W}=0$.
\end{lem}

\begin{proof}
	Since $G\triangleright (h,hg)=h^G\times(hg)^G$, where $\triangleright$
	denotes the diagonal action, we conclude that $X_2^{V,W}=\K G\varphi_2(v\otimes w')$.
	Thus it is enough to prove that $\varphi_2(v\otimes w')=0$. We compute:
	\begin{align*}
		\varphi_2(v\otimes w')=v\otimes w'+\rho(\epsilon)^2gv\otimes hw'&+hv\otimes\varphi_1(v\otimes w)\\
		&-\rho(\epsilon)^{-1}hgv\otimes\varphi_1(v\otimes hw).
	\end{align*}
	Since $\varphi_1$ is a $G$-module map, by acting with $h$ on
	$w'=\varphi_1(v\otimes w)$ we obtain that $\varphi_1(v\otimes hw)=-hw'$.
	Since $hgv=-\rho(\epsilon)^{-1}gv$, the claim follows.
\end{proof}

\begin{rem} \label{rem:G4:X1VW=X1WV}
	The braiding $c_{V,W}$ induces an isomorphism of Yetter-Drinfeld modules
	over $G$ between $X_1^{V,W}$ and $X_1^{W,V}$. The action of $G$ on
	$X_1^{W,V}$ can be obtained from the action of $G$ on $X_1^{V,W}$ in Remark
	\ref{rem:G4:G_on_X1}.
	The element $v'\coloneqq c_{V,W}(w')\in(W\otimes V)_{hg}$ is non-zero. Moreover,
	\begin{equation}
		\label{eq:G4:v'}
		v'=hw\otimes v-\rho(\epsilon)^{-1}\sigma(\epsilon^{-1}h^2)w\otimes gv=\varphi_1(hw\otimes v).
	\end{equation}
\end{rem}

To compute $X_2^{W,V}$ we need the following lemma.

\begin{lem}
	Assume that $\sigma(g)=\rho(h)=-1$ and
	$\rho(\epsilon)=\rho(g^2)\sigma(\epsilon^{-1}h^2)$. 
	Then the following hold:
	\begin{align}
          \label{eq:G4:phi1(ew,v)}
          \varphi_1(\epsilon w\otimes v)=&\;-\sigma(\epsilon h^{-2})hv',\\
          \label{eq:G4:phi1(ew,gv)}
          \varphi_1(\epsilon w\otimes gv)=&\;
            -\sigma(\epsilon h^{-2})\rho(\epsilon^2)\epsilon v',\\
          \label{eq:G4:phi1(w,v)}
          \varphi_1(w\otimes v)=&\;-\sigma(\epsilon^{-1}h^{-2})
            \rho(\epsilon)^{-1}\epsilon hv',\\
          \label{eq:G4:phi1(w,gv)}
          \varphi_1(w\otimes gv)=&\;
            -\rho(\epsilon)\sigma(\epsilon h^{-2})v'.
	\end{align}
\end{lem}

\begin{proof}
	Since $v'=\varphi_1(hw\otimes v)$, acting on this element with $h$ we obtain
	Equation \eqref{eq:G4:phi1(ew,v)}.  Acting on \eqref{eq:G4:phi1(ew,v)} with
	$g$ we obtain Equation \eqref{eq:G4:phi1(ew,gv)}. To prove Equation
	\eqref{eq:G4:phi1(w,v)} act with $\epsilon$ on Equation
	\eqref{eq:G4:phi1(ew,v)}.  Finally, to prove Equation
	\eqref{eq:G4:phi1(w,gv)} act with $g$ on Equation \eqref{eq:G4:phi1(w,v)}.
\end{proof}

\begin{lem}
	\label{lem:G4:Y2}
	Assume that $\sigma(g)=\rho(h)=-1$ and
	$\rho(\epsilon)=\rho(g^2)\sigma(\epsilon^{-1}h^2)$.  Then $X_2^{W,V}$ is
	absolutely simple if and only if $\rho(\epsilon^2)=-1$. In this case,
	$X_2^{W,V}\simeq M(\epsilon hg^2,\rho_2)$, where $\rho_2$ is the irreducible
	character of $G^{\epsilon hg^2}=G^h$ given by
	\begin{align*}
		\rho_2(h)=\rho (\epsilon g^{-2}),&&
		\rho_2(\epsilon^{-1}h^2)=\rho (\epsilon g^{-4}),&&
		\rho_2(g^2)=\rho(g^2).
	\end{align*}
	Moreover, $\rho_2(\epsilon hg^2)=-1$ and the set
	$\{v''\coloneqq\varphi_2(\epsilon w\otimes v'),gv''\}$ is a basis of
	$X_2^{W,V}$. The degrees of these basis vectors are $\epsilon hg^2$ and
	$\epsilon^2 hg^2$, respectively. 
\end{lem}

\begin{proof}
	Since $G\triangleright (g,hg)\cup G\triangleright(\epsilon^2g,hg)=g^G\times
	(hg)^G$, we conclude that 
	\[
		X_2^{W,V}=\K G\{\varphi_2(w\otimes v'),\varphi_2(\epsilon w\otimes v')\}.
	\]
	We first prove that $\varphi_2(w\otimes v')=0$. 
	First, one obtains from
	  Remarks~\ref{rem:G4:X1VW=X1WV} and \ref{rem:G4:G_on_X1},
		that $gv'=-\sigma_1(\epsilon^{-1}h^{-2})\epsilon hv'$. Moreover,
	\begin{align*}
		c_{X_1,W}c_{W,X_1}(w\otimes v')&=c_{X_1,W}(gv'\otimes w)=-\sigma(\epsilon^2)\epsilon hw\otimes gv'\\
		&=\sigma(\epsilon^2)\sigma_1(\epsilon^{-1}h^{-2})\epsilon hw\otimes \epsilon hv'\\
		&=\rho (\epsilon ^2g^2)\epsilon	hw\otimes \epsilon hv'
	\end{align*}
	by Lemma~\ref{lem:G4:X1}. Therefore
	\begin{align*}
          \varphi_2(w\otimes v')=&\;w\otimes v'
          -\rho(\epsilon^2g^2)\epsilon hw\otimes\epsilon hv'\\
		&\;+ghw\otimes\varphi_1(w\otimes v)
		-\rho(\epsilon)^{-1}\sigma(\epsilon^{-1}h^2)gw\otimes\varphi_1(w\otimes gv).
	\end{align*}
	Using Equations \eqref{eq:G4:phi1(w,v)} and \eqref{eq:G4:phi1(w,gv)} we
	conclude that
	\begin{align*}
		\varphi_2(w\otimes v')=&\;w\otimes v'
		-\rho(\epsilon^2g^2)\epsilon hw\otimes\epsilon hv'\\
		&+\rho(\epsilon)^{-1}\sigma(\epsilon h^{-2})\epsilon hw\otimes \epsilon hv'
		-w\otimes v',
	\end{align*}
	and hence $\varphi_2(w\otimes v')=0$. 
	
	Now we use Equation
	\eqref{eq:G4:v'} to compute:
	\begin{align*}
		\varphi_2(\epsilon w\otimes v')=&\;\epsilon w\otimes v'-c_{X_2,W}c_{W,X_2}(\epsilon w\otimes v')\\
		&\;-\epsilon hw\otimes\varphi_1(\epsilon w\otimes v)
                +\rho(\epsilon)^{-1}\sigma(\epsilon h^2)w\otimes\varphi_1(\epsilon w\otimes gv).
	\end{align*}
	Using Lemma \ref{lem:c^2} and Equations \eqref{eq:G4:phi1(ew,v)} and
	\eqref{eq:G4:phi1(ew,gv)} we conclude that
	\begin{equation}        
		\label{eq:G4:v''}
		\begin{aligned}
			v''=\epsilon w\otimes v'&-\sigma (\epsilon ^2)hw\otimes \rho (g^2)\epsilon hv'\\
			&+\sigma(\epsilon h^{-2})\epsilon hw\otimes hv'-\sigma(\epsilon^2)\rho(\epsilon) w\otimes\epsilon v'
		\end{aligned}
	\end{equation}
	belongs to $(W\otimes W\otimes V)_{\epsilon hg^2}$ and it is non-zero. Since
        $G^{\epsilon hg^2}=G^h=hZ(G)$, the
	module $X_2^{W,V}$ is absolutely simple if and only if $hv''\in \K v''$.
	This is equivalent to $\rho(\epsilon^2)=-1$. Then $hv''=\sigma(\epsilon^{-1}h^2)v''=\rho (\epsilon g^{-2})v''$. 
\end{proof}

\begin{rem}
	\label{rem:G4:G_on_Y2}
	Assume that $\sigma(g)=\rho(h)=-1$,
	$\rho(\epsilon)=\rho(g^2)\sigma(\epsilon^{-1}h^2)$,
        and $\rho (\epsilon ^2)=-1$.
	Then the action of $G$
	on $X_2^{W,V}$ is given by:
	\begin{center}
		\begin{tabular}{c|cc}
			$X_2^{W,V}$ & $v''$ & $gv''$ \tabularnewline
			\hline
			$\epsilon$ & $\rho(\epsilon)v''$ & $\rho(\epsilon^{-1})gv''$ \tabularnewline
			$h$ & $\rho (\epsilon g^{-2})v''$ & $\rho(g^{-2})gv''$ \tabularnewline
			$g$ & $gv''$ & $\rho(g^2)v''$ \tabularnewline
		\end{tabular}
	\end{center}
\end{rem}

\begin{lem}
	\label{lem:G4:Y3_auxiliar}
	Assume that $\sigma(g)=\rho(h)=-1$,
	$\rho(\epsilon)=\rho(g^2)\sigma(\epsilon^{-1}h^2)$, and $\rho (\epsilon ^2)=-1$.
	Then the following hold:
	\begin{align}
		\label{eq:G4:phi2(w,v')}\varphi_2(w\otimes v')=&\;0,\\
		\label{eq:G4:phi2(w,ehv')}\varphi_2(w\otimes\epsilon hv')=&\;0,\\
		\label{eq:G4:phi2(w,ev')}\varphi_2(w\otimes\epsilon v')=&\;
    \rho (\epsilon )\sigma(\epsilon^2)v'',\\
		\label{eq:G4:phi2(w,hv')}\varphi_2(w\otimes hv')=&\;
    \rho (\epsilon ^{-1}g^{-2})gv''.
	\end{align}
\end{lem}

\begin{proof}
  In the proof of Lemma~\ref{lem:G4:Y2} we have shown that $\varphi _2(w\otimes
  v')=0$.
	Act with $g$ on this equation to obtain Equation \eqref{eq:G4:phi2(w,ehv')}.
  To prove Equations \eqref{eq:G4:phi2(w,ev')} and
  \eqref{eq:G4:phi2(w,hv')}, act
	with $h^2$ and $g\epsilon $ on $v''=\varphi_2(\epsilon w\otimes v')$,
  respectively.
\end{proof}

\begin{lem}
  \label{lem:G4:Y3}
	Assume that $\sigma(g)=\rho(h)=-1$,
	$\rho(\epsilon)=\rho(g^2)\sigma(\epsilon^{-1}h^2)$ and $\rho(\epsilon^2)=-1$.
	Then $X_3^{W,V}=0$. 
\end{lem}

\begin{proof}
  Since $G\triangleright (g,\epsilon hg^2)=g^G\times (\epsilon hg^2)^G$,
  we conclude that
  $$X_3^{W,V}=\K G\varphi_3(w\otimes v'').$$
	Thus it is enough to prove that $\varphi_3(w\otimes v'')=0$.
  {}From Lemma~\ref{lem:G4:Y3_auxiliar} we know that
	$\varphi_2(w\otimes v')=\varphi_2(w\otimes \epsilon hv')=0$.
	Hence Lemma \ref{lem:c^2} implies that 
	\begin{align*}
		\varphi_3(w\otimes v'')=w\otimes v''-\sigma(\epsilon^2)hw\otimes gv''
    &+\sigma(\epsilon h^{-2})g\epsilon hw\otimes\varphi_2(w\otimes hv')\\
		&-\sigma(\epsilon^2)\rho(\epsilon)gw\otimes\varphi_2(w\otimes\epsilon v').
	\end{align*}
	Now Equations \eqref{eq:G4:phi2(w,ev')} and \eqref{eq:G4:phi2(w,hv')} imply
	that $\varphi_3(w\otimes v'')=0$. 
\end{proof}

We summarize the results concerning the adjoints actions in the following
proposition.

\begin{pro}
	\label{pro:G4:cartan_matrix}
    Assume that 
    \[
      \rho(h)=\sigma(g)=-1,\quad
      \rho(\epsilon)=\rho(g^2)\sigma(\epsilon^{-1}h^2).
    \]
    Then the following hold:
    \begin{enumerate}
        \item $(\ad V)(W)$ is absolutely simple and $(\ad V)^2(W)=0$.
        \item The Yetter-Drinfeld modules
			$(\ad W)^m(V)$ are absolutely simple or zero for all
			$m\in\N_0$ if and only if 
            $\rho(\epsilon^2)=-1$. 
            In this case, 
			$(\ad W)^2(V)\not=0$ and $(\ad W)^3(V)=0$.  
	\end{enumerate}
\end{pro}

\begin{proof}
    The claim follows from Lemmas \ref{lem:G4:X1}, \ref{lem:G4:X2}, \ref{lem:G4:Y2}
    and \ref{lem:G4:Y3}.
\end{proof}

\begin{rem}
    \label{rem:R1}
    Let $\epsilon_1\coloneqq\epsilon^{-1}$, $h_1\coloneqq h^{-1}$ and $g_1\coloneqq
    hg$. 
    Then 
    $G=\langle\epsilon_1,h_1,g_1\rangle$, $\epsilon _1^2\not=1$, and  
    there is a unique group homomorphism $\Gamma_4\to G$ such that
    \[
    a\mapsto g_1,\quad
    b\mapsto h_1,\quad
    \nu \mapsto\epsilon_1.
    \]
\end{rem}

\begin{lem}
	\label{lem:G4:R1}
    Assume that 
    $\sigma(g)=\rho (h)=-1$,
    $\rho(\epsilon)=\rho(g^2)\sigma(\epsilon^{-1}h^2)$, and
    $\rho(\epsilon^2)=-1$.
    Then $R_1(V,W)=\left(V^*,X_1^{V,W}\right)$, where $V^*\simeq M(h_1,\rho^*)$ and
    $\rho^*$ is the irreducible representation of $G^{h}$ dual to $\rho$,
    $X_1^{V,W}\simeq M(g_1,\sigma_1)$ and $\sigma_1$ is the irreducible
    representation of $G^{g_1}$ given in Lemma \ref{lem:G4:X1}, and 
    \begin{align}
        \label{eq:G4:conditions123:R1}&
        \sigma_1(g_1)=\rho^*(h_1)=-1,\quad
        \rho^*(\epsilon_1)=\rho^*(g_1^2)\sigma_1(\epsilon_1^{-1} h_1^{2}),
        \quad
        \rho^*(\epsilon_1^{2})=-1.
     \end{align}
\end{lem}

\begin{proof}
    The description of $R_1(V,W)$ follows from Proposition
    \ref{pro:G4:cartan_matrix}(1) and Lemma \ref{lem:G4:X1}.
    Further,
    \begin{align*}
      \rho^*((hg)^2)\sigma_1(\epsilon h^{-2})&=\rho^*(\epsilon^{-1}h^2g^2)
      \sigma(\epsilon h^{-2})\rho(\epsilon h^{-2})\\
      &=\rho(\epsilon)\rho(g^{-2})\sigma(\epsilon h^{-2})\rho(\epsilon )
      =\rho(\epsilon).
    \end{align*}
    The remaining equations in \eqref{eq:G4:conditions123:R1} are easily
    shown.
\end{proof}

\begin{rem}
    \label{rem:R2}
    Let $\epsilon_2\coloneqq\epsilon^{-1}$, $h_2\coloneqq\epsilon hg^2$ and
    $g_2\coloneqq g^{-1}$.  Then 
    $G=\langle\epsilon_2,h_2,g_2\rangle$, $\epsilon_2^2\not=1$, and
    there is a unique group homomorphism $\Gamma_4\to G$ such that
    \[
    a\mapsto g_2,\quad
    b\mapsto h_2,\quad
    \nu \mapsto\epsilon_2.
    \]
\end{rem}

\begin{lem}
	\label{lem:G4:R2}
    Assume that 
    $\sigma(g)=\rho (h)=-1$,
    $\rho(\epsilon)=\rho(g^2)\sigma(\epsilon^{-1}h^2)$, and
    $\rho(\epsilon^2)=-1$.
    Then $R_2(V,W)=\left(X_2^{W,V},W^*\right)$, where $W^*\simeq
    M(g_2,\sigma^*)$ and $\sigma^*$ is the irreducible representation of
    $G^g$ dual to $\sigma$, $X_2^{W,V}\simeq M(h_2,\rho_2)$ and
    $\rho_2$ is the irreducible representation of $G^{h_2}$ given in
    Lemma \ref{lem:G4:Y2}, and 
    \begin{align}
        \label{eq:G4:condition:R2}&\sigma^*(g_2)=\rho_2(h_2)=-1,\quad
        \rho_2(\epsilon_2)=\rho_2(g_2^2)\sigma^*(\epsilon^{-1}_2h_2^2),\quad
        \rho_2(\epsilon^{2}_2)=-1.
    \end{align}
\end{lem}

\begin{proof}
    The description of $R_2(V,W)$ follows from Proposition
    \ref{pro:G4:cartan_matrix} and Lemma \ref{lem:G4:Y2}.
    Further, $\rho _2(h_2)=-1$ by Lemma~\ref{lem:G4:Y2}, and
    \begin{align*}
        \rho_2(\epsilon_2)=(\rho_2(\epsilon^{-1}h^2)\rho_2(h)^{-2})
        =\rho(\epsilon g^{-4}\epsilon ^{-2}g^4)=\rho(\epsilon ^{-1}),
    \end{align*}
    \begin{align*}
        \rho_2(g^{-2})\sigma^*(\epsilon^{-1}h^2g^4)
        =\rho(g^{-2})\sigma(\epsilon h^{-2})=\rho(\epsilon)^{-1}.
    \end{align*}
    Now one easily concludes the claimed formulas
    on $\rho _2$.
\end{proof}

Before proving Theorem \ref{thm:G4} we list some well-known finite-dimensional
Nichols algebras related to non-abelian epimorphic images of $\Gamma_4$. 

\begin{pro}
	\label{pro:NA_small}
	Let $G$ be a non-abelian quotient of $\Gamma_4$. Let $V=M(h,\rho)$ and
	$W=M(g,\sigma)$, where $\rho$ and $\sigma$ are characters of the
  centralizers $G^h$ and
  $G^g$, respectively. Assume that $\sigma(g)=\rho (h)=-1$,
	$\rho(\epsilon)=\rho (g^2)\sigma(\epsilon^{-1}h^2)$, and
	$\rho(\epsilon)^2=-1$. Then $V$ and $(\ad W)^2(V)$ are of diagonal type.
  The braiding matrices
  with respect to the bases $\{v,gv\}$ and $\{v'',gv''\}$ are
	\[
	\begin{pmatrix}
		-1 & \rho(\epsilon)\\
		\rho(\epsilon) & -1
	\end{pmatrix}
	\quad\text{and}\quad
	\begin{pmatrix}
    -1 & \rho(\epsilon^{-1})\\
    \rho(\epsilon^{-1}) & -1
	\end{pmatrix},
	\]
	respectively. In particular, The Nichols algebras $\NA(V)$ and
  $\NA\left( (\ad W)^2(V) \right)$ are of Cartan type $A_1\times A_1$ if
  $\mathrm{char}\,\K=2$ and $A_2$ if $\mathrm{char}\,\K\not=2$.
  Their Hilbert series is
	\[
	\mathcal{H}_{\NA(V)}(t)=\mathcal{H}_{\NA((\ad W)^2(V))}(t)=
	\begin{cases}
		(1+t)^2 & \text{if $\mathrm{char}\,\K=2$,}\\
		(1+t)^2(1+t^2) & \text{if $\mathrm{char}\,\K\ne2$}.
	\end{cases}
	\]
\end{pro}

\begin{proof}
	The braiding matrices are obtained from a direct calculation using Remarks
	\ref{rem:G4:G_on_V} and \ref{rem:G4:G_on_Y2}.  The claim concerning the
	Hilbert series follows from the definition of the root system \cite[Section
	3]{MR2207786} and \cite[Thm.\,1]{MR2207786}.
\end{proof}

\begin{pro}
	\label{pro:NA_64}
	Let $G$ be a non-abelian quotient of $\Gamma_4$. Let $V=M(h,\rho)$ and
	$W=M(g,\sigma)$, where $\rho$ and $\sigma $ are characters of the
  centralizers $G^h$ and $G^g$, respectively.
  Assume that $\sigma(g)=\rho(h)=-1$,
	$\rho(\epsilon)=\rho (g^2)\sigma(\epsilon^{-1}h^2)$, and
	$\rho(\epsilon^2)=-1$. Then the Nichols algebras of $W$ and $(\ad
	W)(V)$ are finite-dimensional with Hilbert series 
	\begin{align*}
		\mathcal{H}_{\NA(W)}(t)&=\mathcal{H}_{\NA((\ad W)(V))}(t)=(1+t)^4(1+t^2)^2\\
		&=1+4t+8t^2+12t^3+14t^4+12t^5+8t^6+4t^7+t^8.
	\end{align*}
\end{pro}

\begin{proof}
	Let $H$ be the subgroup of $G$ generated by $\supp W$.
	Then $H=\langle g,\epsilon \rangle $, and there exists a unique
	surjective group homomorphism $\Gamma_2\to H$ with $a\mapsto g$, $b\mapsto\epsilon
	g$, and $\nu\mapsto\epsilon^2$. Consider $W$ as Yetter-Drinfeld module over
  $H$ by restriction of the $G$-module structure to $H$.
  Since $g^H=\{g,\epsilon^2g\}$ and $(\epsilon
	g)^H=\{\epsilon g,\epsilon^3g\}$, we conclude that $W=V'\oplus W'$,
  where $V'=\K w+\K \epsilon w$ and $W'=\K hw +\K \epsilon hw$
  are simple Yetter-Drinfeld modules over $H$.
  Further, $V'\simeq M(g,\rho')$, for some character $\rho '$ of
  $H^g=\langle g,\epsilon^2\rangle$, and $W'\simeq M(\epsilon g,\sigma')$
  for some character $\sigma '$ of $H^{\epsilon g}=\langle \epsilon
	g,\epsilon^2\rangle$.  Using Remark~\ref{rem:G4:G_on_W}
  one obtains the following formulas.
  \begin{align}
    \label{eq:rho':W}
    \rho '(g)=&\;-1,& \rho '(\epsilon ^2)=&\;\sigma (\epsilon ^2),&
    \sigma '(\epsilon g)=&\;-1,& \sigma '(\epsilon ^2)=&\;\sigma (\epsilon ^2).
  \end{align}
  Therefore $\rho '(\epsilon ^2(\epsilon g)^2)\sigma '(\epsilon
  ^2g^2)=\sigma (\epsilon ^4)=1$, and hence
	$W=V'\oplus W'$ satisfies the assumptions of \cite[Thm.\,4.6]{MR2732989}.
  Thus $\NA (W)$ is finite-dimensional and has the claimed
  Hilbert series.

	The claim concerning the Nichols algebra $\NA\left((\ad W)(V) \right)$ is
  similar. We may replace $(\ad W)(V)$ by $X_1^{V,W}$.
  Let $L$ be the subgroup of $G$ generated by $\supp X_1^{V,W}$ and
	consider the unique group homomorphism $\Gamma_2\to L$ with $a\mapsto hg$,
	$b\mapsto\epsilon hg$ and $\nu\mapsto\epsilon^2$. As we did in the previous
	paragraph, \cite[Thm.\,4.6]{MR2732989} yields the Hilbert series of
  $\NA(X_1^{V,W})$. 
\end{proof}

\begin{proof}[Proof of Theorem \ref{thm:G4}]
    $(1)\implies(2)$.  Since $\NA(V\oplus W)$ is finite-dimensional,
    the pair $(V,W)$ admits
    all reflections by \cite[Cor 3.18]{MR2766176} and the Weyl groupoid is
    finite by \cite[Prop.\,3.23]{MR2766176}. 

    $(2)\implies(3)$. By \cite[Prop.\,4.3]{partII}, after changing the object of
    $\mathcal{W}(V,W)$ and possibly interchanging $V$ and $W$, we may assume
    that $V=M(h,\rho)$ and $W=M(g,\sigma)$, such that $\deg \rho =\deg \sigma
    =1$, $(\ad V)^2(W)=0$ and $(\ad W)^4(V)=0$. By
    \cite[Thm.\,7.2(3)]{MR2734956}, $(\ad W)^m(V)$ is absolutely simple or zero
    for all $m\in\N_0$. Lemma~\ref{lem:G4:rho(g)=-1} implies that $(\ad W)(V)$
    and $(\ad W)^2(V)$ are non-zero and $\rho(h)=\sigma(g)=-1$. Since $(\ad
    V)(W)$ and $(\ad W)^2(V)$ are absolutely simple,
    Lemmas~\ref{lem:G4:X1} and \ref{lem:G4:Y2} imply that
    $\rho(\epsilon)=\rho(g^2)\sigma(\epsilon^{-1}h^2)$ and
    $\rho(\epsilon^2)=-1$.

    $(3)\implies(1)$. Proposition~\ref{pro:G4:cartan_matrix} and
    Lemmas~\ref{lem:G4:R1} and \ref{lem:G4:R2} imply that $\mathcal{W}(V,W)$
    is standard with Cartan matrix of type $B_2$. By \cite[Cor.\,2.7(2)]{MR2732989}, 
    \[
        \NA(V\oplus W)\simeq\NA(V)\otimes\NA((\ad W)(V))\otimes\NA((\ad W)^2(V))\otimes\NA(W)
    \]
    as $\N_0^2$-graded vector spaces, where 
    \begin{gather*}
        \deg V=\alpha_1,\quad
        \deg W=\alpha_2,\\
        \deg\left((\ad W)(V)\right)=\alpha_1+\alpha_2,\quad
        \deg\left((\ad W)^2(V)\right)=\alpha_1+2\alpha_2.
    \end{gather*}
		Now the claim on the Hilbert series of $\NA (V\oplus W)$ follows from
    Propositions \ref{pro:NA_small} and \ref{pro:NA_64}.
\end{proof}

\section{An application}
\label{section:applications}

In \cite{partII} we presented five quandles which are essential to our
classification.  These quandles are:
\begin{equation}
    \label{eq:quandles}
    \begin{aligned}
        Z_T^{4,1} &:\;(243)\;(134)\;(142)\;(123)\;\id\\
        Z_2^{2,2} &:\;(24)\;(13)\;(24)\; (13)\\
        Z_3^{3,1} &:\;(23)\;(13)\;(12)\; \id\\
        Z_3^{3,2} &:\;(23)(45)\;(13)(45)\;(12)(45)\;(123)\;(132)\\
        Z_4^{4,2} &:\;(24)(56)\;(13)(56)\;(24)(56)\;(13)(56)\;(1234)\;(1432)
    \end{aligned}
\end{equation}

Let us describe these quandles in a different way.  Remark \ref{rem:T+1} states
that $Z_T^{4,1}$ is isomorphic to the disjoint union of the trivial quandle
with one element and the quandle associated with the vertices of the
tetrahedron.  The quandle $Z_2^{2,2}$ is isomorphic to the dihedral quandle
$\D_4$ with four elements. The quandle $Z_3^{3,1}$ is isomorphic to the
disjoint union of the trivial quandle with one element and the dihedral quandle
$\D_3$ with three elements. Remark \ref{rem:amalgamated} describes the quandle
$Z_4^{4,2}$ as an amalgamated sum of $\D_4$ and the trivial quandle with two
elements.  Similarly, the quandle $Z_3^{3,2}$ can be presented as an
amalgamated sum of $\D_3$ with the trivial quandle of two elements. See
\cite[\S1]{MR1994219} for disjoint union and amalgamated sum of quandles.

\begin{thm} 
	\label{thm:main}
    Let $\K$ be a field, $G$ be a non-abelian group, and $V$ and $W$ be
    finite-dimensional absolutely simple Yetter-Drinfeld modules over $G$ such
    that $G$ is generated by the support of $V\oplus W$. 
		Assume that the pair $(V,W)$
    admits all reflections and the Weyl groupoid $\mathcal{W}(V,W)$ of $(V,W)$
    is finite.  If $c_{W,V}c_{V,W}\ne \id_{V\otimes W}$, then
    $\supp(V\oplus W)$ is isomorphic to one of the following quandles:
		\[
			Z_T^{4,1},\;
			Z_2^{2,2},\;
			Z_3^{3,1},\;
			Z_3^{3,2},\;
			Z_4^{4,2}.
		\]
    Moreover, the group $G$ is isomorphic to a quotient of the enveloping group
    of the quandle $\supp(V\oplus W)$:
    \begin{center}
        \begin{tabular}{c|c c c c c}
            Quandle &
            $Z{}_{T}^{4,1}$ &
            $Z_{2}^{2,2}$ &
            $Z_{3}^{3,1}$ &
            $Z_{3}^{3,2}$ &
            $Z_{4}^{4,2}$
            \tabularnewline
            \hline 
            Enveloping group &
            $T$ &
			$\Gamma_{2}$ &
            $\Gamma_{3}$ &
            $\Gamma_{3}$ &
            $\Gamma_{4}$
        \end{tabular}
    \end{center}
\end{thm}

\begin{proof}
	By \cite[Prop.\,4.3]{partII}, after
	changing the object of $\mathcal{W}(V,W)$ and possibly interchanging $V$ and
	$W$ we may assume that $(\ad V)(W)\ne0$, $(\ad V)^2(W)=0$, and $(\ad
  W)^4(V)=0$. Then
	\cite[Thm.\,4.4]{partII} implies that the group $G$ is a quotient of $\Gamma_n$ for
	$n\in\{2,3,4\}$ or a quotient of $T$. 

	Suppose first that $G$ is a quotient of $\Gamma_2$. Since $\Gamma_2$ has conjugacy
	classes of size one or two \cite[Section 3]{MR2732989} and $G$ is
	non-abelian, it follows that the quandles appearing after applying
	reflections to $(V,W)$ are isomorphic to $\supp(V\oplus W)\simeq Z_2^{2,2}$. 

	Suppose now that $G$ is a quotient of $\Gamma_3$. Any conjugacy class of $G$
  has size $1$, $2$, or $3$.  Assume that an object of $\mathcal{W}(V,W)$
  is represented by a pair $(V',W')$ of absolutely irreducible Yetter-Drinfeld
  modules over $G$ with $|\supp V'|=|\supp W'|=3$. Then
  $\supp(V'\oplus W')$ is isomorphic as a quandle to $(gz_1)^G\cup
  (hgz_2)^G$ for some $z_1,z_2\in Z(G)$ by \cite[Sect.\,3.1]{MR2732989}.
  Let $s=gz_1$ and $t=hgz_2$. Since $stst\ne tsts$,
  \cite[Prop.\,8.5]{MR2734956} implies that $(\ad V')(W')$ is not irreducible,
  which contradicts \cite[Thm.\,2.5]{MR2732989}. 
	
	Suppose that $G$ is a quotient of $\Gamma_4$. Then all conjugacy classes of $G$
  have size $1$, $2$, or $4$ \cite[Section 3]{MR2732989}. After changing the object
	of $\mathcal{W}(V,W)$ we may assume that $V=M(h,\rho)$, where $\rho$ is a
	character of $G^h$ and $W=M(g,\sigma)$, where $\sigma$ is a character of
	$G^g$. In particular, $\supp(V\oplus W)\simeq Z_4^{4,2}$ as quandles. By
	Theorem \ref{thm:G4}, $\rho(h)=\sigma(g)=-1$,
	$\rho(\epsilon)=\rho(g^2)\sigma(\epsilon^{-1}h^2)$ and $\rho(\epsilon)^2=-1$.
	Lemmas \ref{lem:G4:R1} and \ref{lem:G4:R2} imply that after applying
	reflections to $(V,W)$ one obtains new pairs $(V',W')$ such that
	$\supp(V'\oplus W')\simeq Z_4^{4,2}$ as quandles. 

	Finally, suppose that $G$ is a quotient of $T$. By changing the object
  of $\mathcal{W}(V,W)$ if necessary, we may
	assume that $V=M(z,\rho)$, where $\rho$ is a representation of $G$ and
	$W=M(x_1,\sigma)$, where $\sigma$ is a character of $G^{x_1}$.
        In particular $\supp(V\oplus W)\simeq Z_T^{4,1}$ as quandles.
        From Theorem
	\ref{thm:T} we obtain that $\deg\rho=1$,
	$(\rho(x_1)\sigma(z))^2-\rho(x_1)\sigma(z)+1=0$,
	$\sigma(x_1)=-1$, $\sigma(x_2x_3)=1$, and $\rho(x_1z)\sigma(z)=1$.
        Lemmas \ref{lem:T:R1} and \ref{lem:T:R2}
	imply that all reflections of $(V,W)$ are pairs $(V',W')$ with
	$\supp(V'\oplus W')\simeq Z_T^{4,1}$ as quandles. This proves the theorem.
\end{proof}

As a combination of our results with the main results in \cite{MR2766176} we
obtain the following corollary concerning finite-dimensional Nichols algebras.

\begin{cor} \label{cor:main}
    Let $\K$ be a field, $G$ be a non-abelian group, and $V$ and $W$ be
    absolutely simple Yetter-Drinfeld modules over $G$ such
    that $G$ is generated by $\supp(V\oplus W)$.  Assume that the Nichols
    algebra $\NA(V\oplus W)$ is finite-dimensional. If
    $c_{W,V}c_{V,W}\ne \id _{V\otimes W}$, then $\supp(V\oplus W)$ is
    isomorphic to one of the quandles of Theorem \ref{thm:main} and the group
    $G$ is isomorphic to a quotient of the enveloping group of the quandle
    $\supp(V\oplus W)$.
\end{cor}

\begin{proof}
    Since $\NA(V\oplus W)<\infty$, the pair $(V,W)$ admits all reflections by
    \cite[Cor 3.18]{MR2766176} and the Weyl groupoid is finite by
    \cite[Prop.\,3.23]{MR2766176}. So Theorem \ref{thm:main} applies.
\end{proof}

\subsection*{Acknowledgement}
Leandro Vendramin was supported by Conicet, UBACyT 20020110300037 and the
Alexander von Humboldt Foundation.

\bibliographystyle{abbrv}
\bibliography{refs}

\end{document}